\documentclass[12pt, reqno]{amsart}
\usepackage{mathtools}
\usepackage{amssymb,amsmath,dsfont}
\usepackage[table,xcdraw]{xcolor}
\usepackage{pgfplots}
\pgfplotsset{compat=1.14}
\usepackage{enumitem}
\usepackage{calc}
\usepackage{hyperref}
\usepackage{comment}

\usepackage[margin=1in]{geometry}

\usepackage{booktabs}
\usepackage{caption} 
\newtheorem{thm}{Theorem}[section]
\newtheorem{theorem}[thm]{Theorem}
\newtheorem{lemma}[thm]{Lemma}
\newtheorem{lem}[thm]{Lemma}
\newtheorem{prop}[thm]{Proposition}

\newtheorem{cor}[thm]{Corollary}

\newtheorem{claim}[thm]{Claim}

\theoremstyle{plain}

\newcommand{\Z}{\mathbb{Z}}
\newcommand{\ep}{\varepsilon}
\newcommand{\R}{\mathbb{R}}
\newcommand{\N}{\mathbb N}

\DeclareMathOperator{\cov}{Cov}
\DeclareMathOperator{\var}{Var}
\renewcommand{\P}{\mathbb{P}}
\newcommand{\E}{\mathbb{E}}
\newcommand{\cE}{\mathcal{E}}

\newcommand{\eqd}{\,{\buildrel d \over =}\,}
\DeclareMathOperator{\GE}{GE}

\DeclareMathOperator{\std}{Std}
\DeclareMathOperator{\med}{Med}
\DeclareMathOperator{\supp}{Supp}
\DeclareMathOperator{\DE}{DE}

\DeclareMathOperator{\Cov}{Cov}

\newcommand{\eps}{\varepsilon}

\newtheorem*{assumption*}{\assumptionnumber}
\providecommand{\assumptionnumber}{}


\newtheorem*{disorder*}{\assumptionnumber}
\providecommand{\assumptionnumber}{}


\title{Minimal surfaces in strongly correlated random environments}

\author{Barbara Dembin}
\address{Barbara Dembin\hfill\break
    IRMA, CNRS et Université de Strasbourg, Strasbourg,France}
\email{barbara.dembin@math.unistra.fr}
\author{Dor Elboim}
\address{Dor Elboim\hfill\break
    Department of Mathematics,
    Stanford University,
    California, United States.}
\email{dorelboim@gmail.com}
\author{Ron Peled}
\address{Ron Peled\hfill\break Department of Mathematics, University of Maryland, College Park, United States.\hfill\break
School of Mathematical Sciences, Tel Aviv University, Tel Aviv, Israel.}
\email{peledron@tauex.tau.ac.il}
\date{\today}

\begin{document}

\begin{abstract}
A minimal surface in a random environment (MSRE) is a $d$-dimensional surface in $(d+n)$-dimensional space which minimizes the sum of its elastic energy and its environment potential energy, subject to prescribed boundary values. Apart from their intrinsic interest, such surfaces are further motivated by connections with disordered spin systems and first-passage percolation models. In this work, we consider the case of strongly correlated environments, realized by the model of harmonic MSRE in a fractional Brownian environment of Hurst parameter $H\in(0,1)$. This includes the case of Brownian environment ($H=1/2$ and $n=1$), which is commonly used to approximate the domain walls of the $(d+1)$-dimensional random-field Ising model.

We prove that surfaces of dimension $d\in\{1,2,3\}$  delocalize with power-law fluctuations, and determine their precise transversal and minimal energy fluctuation exponents, as well as the stretched exponential exponents governing the tail decay of their distributions. These exponents are found to be the same in all codimensions $n$, depending only on $d$ and $H$. The transversal and minimal energy fluctuation exponents are specified by two scaling relations.

We further show that surfaces of dimension $d=4$ delocalize with sub-power-law fluctuations, with their height and minimal energy fluctuations tied by a scaling relation. Lastly, we prove that surfaces of dimensions $d\ge 5$ localize.

These results put several predictions from the physics literature on mathematically rigorous ground.
\end{abstract}

\maketitle

\setcounter{tocdepth}{1}
\tableofcontents

\section{Introduction}

The mathematical theory of minimal surfaces
has been a topic of enduring interest since the 18th century works of Euler and Lagrange (see~\cite{colding2011course, meeks2012survey, de2022regularity}). In~\cite{dembin2024minimal} we embarked on a mathematical exploration of the properties of minimal surfaces \emph{in a random environment}, i.e., surfaces placed in an inhomogeneous, random media, which minimize the sum of an internal (elastic) energy and an environment potential energy subject to prescribed boundary values. Apart from its intrinsic interest, the study of such surfaces is further motivated by connections with disordered spin systems and first-passage percolation models.

We study $d$-dimensional minimal surfaces in a $(d+n)$-dimensional random environment (termed the disorder) using the model of \emph{harmonic} minimal surfaces in a random environment (see Section~\ref{subsection : the model} below), whose special structure facilitates the analysis.
Our previous work~\cite{dembin2024minimal} was mainly concerned with disorders with short-range correlations (``independent disorder''). Here we consider \emph{strongly correlated random enviroments}, which we realize by fractional Brownian fields of Hurst parameter $H\in(0,1)$. In the special case $H=1/2$ and $n=1$, the disorder becomes a two-sided Brownian motion; a case known in the physics literature as \emph{random-field disorder} and commonly used to approximate the domain walls of the $(d+1)$-dimensional random-field Ising model (see \cite[Section~6.2]{dembin2024minimal}). A limiting case $H=1$ (``linear disorder'') admits an exact solution; see, e.g.,~\cite{dario2023random}  and~\cite[Section 6.3]{dembin2024minimal}.

Our main results show that the minimal surfaces behave very differently in dimensions $d<4$, $d=4$ and $d>4$:
\begin{enumerate}
\item In low dimensions ($d<4$), we prove that the surfaces delocalize with power-law fluctuations. The exponents governing the delocalization are found to be the same for all codimensions $n$, depending only on $d$ and $H$, with the formulas
\begin{equation}\label{eq:exponent predictions}
    \begin{split}
        \xi &:= \frac{4-d}{4-2H},\\
        \chi &:= \frac{4-d}{2-H} + d-2,
    \end{split}
\end{equation}
for the roughness exponent $\xi$ (also called the transversal fluctuation exponent) and the minimal energy fluctuation exponent $\chi$. The exponents are uniquely defined by the following two scaling relations
\begin{align}
    \chi &= 2\xi + d-2,\label{eq:scaling relation}\\
    \chi &= H\xi + \frac{d}{2}.\label{eq:second scaling relation}
\end{align}

We further identify the exponents governing the stretched exponential tail behavior of the height and minimal energy distributions.

\item In the critical dimension ($d=4$), we prove that the surfaces delocalize, but only at a sub-power-law rate. We further show that their height and minimal energy fluctuations are tied together by a version of the scaling relation~\eqref{eq:second scaling relation}, suitably interpreted to account for the sub-power-law fluctuations.
\item In high dimensions ($d>4$), we prove that the surfaces localize and have minimal energy fluctuation exponent $\chi=\frac{d}{2}$. We further identify the exponent in the stretched exponential tail decay of the height distribution, which is shown to be independent of $d$ and $n$, depending only on $H$.
\end{enumerate}

The results put several predictions from the physics literature on rigorous ground (See Section~\ref{sec:predictions}).

\smallskip
The next sections describe the model, our results and relevant predictions from the physics literature.

\subsection{The model}\label{subsection : the model}

Let $d,n\ge 1$ be integers. We model a $d$-dimensional minimal surface in a $(d+n)$-dimensional random environment by the ground configurations (functions of minimal energy) of the following Hamiltonian. The surface is modeled by a function $\varphi:\Z^d\to\R^n$, defined on the cubic lattice $\Z^d$ and having $n$ components. Given an \emph{environment} $\eta:\Z^d\times\R^n\to(-\infty,\infty]$, later taken to be random and termed the \emph{disorder}, and an \emph{environment strength} $\lambda>0$, the \emph{formal Hamiltonian} for $\varphi$ is
\begin{equation}\label{eq:formal Hamiltonian}
    H^{\eta,\lambda}(\varphi):=\frac{1}{2}\sum_{u\sim v}\|\varphi_u - \varphi_v\|^2 + \lambda \sum_{v} \eta_{v,\varphi_v},
\end{equation}
where $\|\cdot\|$ is the Euclidean norm in $\R^n$ and $u\sim v$ indicates that $u,v\in\Z^d$ are adjacent.
Our goal is to study the minimizers of $H^{\eta,\lambda}$ in finite domains with prescribed boundary values.
Given a finite $\Lambda\subset\Z^d$ and a function $\tau:\Z^d\to\R^n$, the \emph{finite-volume Hamiltonian} in~$\Lambda$ is given by
\begin{equation}\label{eq:finite volume Hamiltonian}
H^{\eta,\lambda,\Lambda}(\varphi):=\frac{1}{2}\sum_{\substack{u\sim v\\\{u,v\}\cap\Lambda\neq\emptyset}}\|\varphi_u - \varphi_v\|^2 + \lambda \sum_{v\in\Lambda} \eta_{v,\varphi_v},
\end{equation}
and the \emph{configuration space with boundary values $\tau$ outside $\Lambda$} is given by
\begin{equation}\label{set}
   \Omega^{\Lambda,\tau} := \{\varphi:\Z^d\to\R^n\colon \varphi_v=\tau_v\text{ for $v\in\Z^d\setminus\Lambda$}\}.
\end{equation}
We write $\varphi^{\eta,\lambda,\Lambda,\tau}$ for the \emph{ground configuration} of the finite-volume model, i.e., for the $\varphi\in \Omega^{\Lambda,\tau}$ which minimizes $H^{\eta,\lambda,\Lambda}$ (it is shown in Proposition~\ref{prop:existence} below that for our choice of environment $\eta $ the minimizer exists and is unique). We let
\begin{equation}
\GE^{\eta,\lambda,\Lambda,\tau}:=H^{\eta,\lambda,\Lambda}(\varphi^{\eta,\lambda,\Lambda,\tau})
\end{equation}
be the \emph{ground energy}.

To lighten notation, we shall omit the superscript $\tau$ in the (common) case that $\tau\equiv 0$.

\subsubsection{Background}
We briefly comment on earlier studies of the model~\eqref{eq:formal Hamiltonian}; see~\cite[Section 1.4.1]{dembin2024minimal} for additional details. The physics literature contains extensive references to~\eqref{eq:formal Hamiltonian} (sometimes in the continuum), going back at least to the works of Grinstein--Ma~\cite{grinstein1982roughening} and Villain~\cite{villain1982commensurate} (who considered it in relation with the random-field Ising model); see the reviews~\cite{forgacs1991behavior,giamarchi1998statics, wiese2003functional, giamarchi2009disordered,ferrero2021creep,wiese2022theory}, as well as our own discussion~\cite[Section 6]{dembin2024minimal} of different contexts in which the model arises. In the mathematically rigorous literature, however, the model has so far received limited attention. Besides our earlier study~\cite{dembin2024minimal}, we are only aware of the following:
\begin{enumerate}
	\item Bakhtin et al.~\cite{bakhtin2014space, bakhtin2016inviscid,bakhtin2019thermodynamic, bakhtin2018zero, bakhtin2022dynamic} and Berger--Torri~\cite{berger2019entropy,berger2021beyond} consider the $d=n=1$ case with ``independent disorder''.
	\item There are related studies for $d=1$ and general $n$ of directed Gaussian polymers in a Poissonian environment (see, e.g., the review by Comets--Cosco~\cite{comets2018brownian}).
	\item The case of general dimension $d$ has recently been considered in the limit $n\to\infty$, with a different focus: Ben Arous--Bourgade--McKenna~\cite{ben2024landscape} study its annealed landscape complexity and Ben Arous--Kivimae~\cite{arous2024free, arous2024larkin} obtain a Parisi-type formula for it.
	\item Just prior to the first appearance of this paper, Otto--Palmieri--Wagner~\cite{otto2025minimizing} published a study of the $d=n=1$ case with Brownian disorder ($H=1/2$); see the discussion after Theorem~\ref{thm:main 123}.
\end{enumerate}
In particular, this work and our earlier~\cite{dembin2024minimal} may be the only mathematically rigorous works so far to study the geometry of the surfaces minimizing~\eqref{eq:formal Hamiltonian} in dimensions $d\ge 2$.

\subsubsection{Fractional Brownian disorder}\label{sec:frac}

In this work we focus exclusively on the case where the disorder $\eta$ is independent and identically distributed between vertices, with $t\mapsto\eta_{v,t}$ an \emph{$n$-dimensional fractional Brownian field of Hurst parameter $0<H<1$} for each $v$. We proceed to make this precise.

\smallskip
{\bf Fractional Brownian field}: An $n$-dimensional fractional Brownian field of Hurst parameter $0<H\le 1$ is a (sample-path) continuous, centered, Gaussian process $(B^H_t)_{t\in\R^n}$, normalized to satisfy $B^H_0=0$, whose covariance is described by
\begin{equation}\label{eq:covariance by variance}
	\var(B^H_t - B^H_s) = \|t-s\|^{2H},\qquad t,s\in\R^n.
\end{equation}
The covariance is equivalently described by
\begin{equation}\label{eq:covariance}
	\Cov(B^H_t,B^H_s)=\frac 12(\|t\|^{2H}+\|s\|^{2H}-\|t-s\|^{2H}),\qquad t,s\in\R^n.
\end{equation}
See, e.g.,~\cite{ossiander1989certain} for a short proof that~\eqref{eq:covariance} is indeed positive definite so that such a Gaussian process exists. The existence of a continuous modification is a straightforward consequence of (a generalization of) Kolmogorov's theorem (see, e.g., \cite[Theorem 4.23]{Kallenberg}).

In the above, for random variables $X,Y\in\R^n$, we write $\var(X):=\E(\|X - \E(X)\|^2)$ and $\Cov(X,Y):=\E((X-\E(X))\cdot(Y-\E(Y)))$, with $\cdot$ being the standard inner product in $\R^n$.

\smallskip
{\bf Special cases}: $H=1/2$ and $n=1$: In this case, $B^H$ is a (two-sided) \emph{Brownian motion}.

$H=1$: In this case, $B^H$ degenerates to a \emph{linear function with a random slope}; precisely,
\begin{equation}\label{eq:H=1 fractional Brownian field}
(B_t^{H=1})_{t\in\R^n} \eqd (\zeta\cdot t)_{t\in\R^n},
\end{equation}
where $\zeta$ has the standard Gaussian distribution in $\R^n$. 

\smallskip
{\bf Properties}:
While fractional Brownian fields are not stationary, it is important for our arguments that they have \emph{stationary increments}, in the sense that
\begin{equation}\label{eq:stationary increments for fractional Brownian field}
	\text{for each $s\in\R^n$, the process $(B^H_{t+s} - B^H_s)_{t\in\R^n}$ has the same distribution as $B^H$.}
\end{equation}
We also note the scaling covariance and isometry invariance properties:
\begin{align}
	&\text{for each $\sigma\in\R$, the process $(B^H_{\sigma t})_{t\in\R^n}$ has the same distribution as $|\sigma|^{H} B^H$,}\label{eq:fBM self similarity}\\
	&\text{for each $n\times n$ orthogonal $T$, the process $(B^H_{Tt})_{t\in\R^n}$ has the same distribution as $B^H$.}
\end{align}
These properties are simple consequences of~\eqref{eq:covariance by variance}.

\smallskip
{\bf Fractional Brownian disorder}: As mentioned, we take the disorder $\eta$ to satisfy that $(\eta_{v,\cdot})_{v\in\Z^d}$ are independent, with $t\mapsto\eta_{v,t}$ having the distribution of $B^H$ for each $v$.

We will repeatedly use the following consequence of our choice of $\eta$ and~\eqref{eq:stationary increments for fractional Brownian field}: For $s:\Z^d\to\R^n$, define the \emph{shifted and recentered disorder} $\eta^s:\Z^d\times\R^n\to\R$ by
\begin{equation}\label{eq:eta s def}
	\eta^s_{v,t}:=\eta_{v,t-s_v}-\eta_{v,-s_v}.
\end{equation}
Then, the disorder has stationary increments in the sense that
\begin{equation}\label{eq:distribution of shifted disorder}
	\text{for every $s:\Z^d\to\R^n$, the process $\eta^s$ has the same distribution as $\eta$.}
\end{equation}

In our main results, we restrict to $0<H<1$. The case $H=1$, sometimes termed random-rod model~\cite{grinstein1982roughening,forgacs1991behavior}, Larkin model (following Larkin~\cite{larkin1970effect}; c.f.~\cite{giamarchi2009disordered}) or linear disorder~\cite{dembin2024minimal}, admits an exact solution, as discussed, e.g., in~\cite{dario2023random}  and~\cite[Section 6.3]{dembin2024minimal}. In particular, it is known that for $H=1$, the minimal surfaces delocalize with the exponent formulas~\eqref{eq:exponent predictions} in dimensions $d<4$, delocalize to height $\sqrt{\log L}$ in dimension $d=4$ and are localized in dimensions $d>4$ (with the notation of Section~\ref{sec:main results}).

\subsection{Basic facts} Before stating our main results, we briefly comment on the effects of the disorder strength $\lambda$ and boundary values $\tau$ and explain that we will restrict attention to the case that $\lambda=1$ and $\tau\equiv 0$. We also note that the minimal surface is indeed well defined.

\subsubsection{The dependence on the disorder strength $\lambda$}\label{sec:lambda} The self-similarity of fractional Brownian fields implies that the disorder strength parameter $\lambda$ only enters the model through scaling. To see this, observe that the Hamiltonian~\eqref{eq:finite volume Hamiltonian} satisfies, for each $\sigma\in\R$, the deterministic identity,
\begin{equation}
H^{\eta,\lambda,\Lambda}(\sigma\varphi)=\sigma^2\bigg(\frac{1}{2}\sum_{\substack{u\sim v\\ \{u,v\}\cap\Lambda\neq\emptyset}}\|\varphi_u - \varphi_v\|^2 + \sigma^{-2}\lambda \sum_{v\in\Lambda} \eta_{v,\sigma\varphi_v}\bigg)=\sigma^2 H^{\zeta,1,\Lambda}(\varphi)
\end{equation}
where $(\zeta_{v,t})$ is defined by $\zeta_{v,t}:=\sigma^{-2}\lambda \eta_{v,\sigma t}$. Now, the self-similarity of fractional Brownian fields~\eqref{eq:fBM self similarity}, implies that $\zeta$ has the same distribution as $\eta$ when $\sigma=\lambda^{\frac 1{2-H}}$. Consequently, without loss of generality,
\begin{equation*}
    \text{we set $\lambda = 1$ in the rest of the paper and remove it from all notation.}
\end{equation*}

\subsubsection{Existence and uniqueness} The next proposition shows that the minimal surface is well defined for all finite domains and boundary values.

\begin{prop}\label{prop:existence}
    Let $\Lambda\subset\Z^d$ be finite. Let $\tau:\Z^d\to\R^n$. Almost surely, there exists a unique $\varphi\in\Omega^{\Lambda,\tau}$ such that
    \begin{equation}
        H^{\eta,\Lambda}(\varphi) = \min_{\psi\in\Omega^{\Lambda,\tau}} H^{\eta,\Lambda}(\psi).
    \end{equation}
\end{prop}

\subsubsection{The dependence on the boundary values}
There is a simple relation between the distribution of the minimal surface under general boundary values $\tau$ and the distribution under zero boundary values. The specific form of the Hamiltonian~\eqref{eq:formal Hamiltonian} together with the stationary increments property~\eqref{eq:distribution of shifted disorder} of fractional Brownian disorder yield that for each $\tau$,
\begin{equation}\label{eq:distribution of surface with general boundary values}
    \varphi^ {\eta,\Lambda,\tau}\eqd \varphi^ {\eta,\Lambda}+\bar{\tau}^\Lambda
\end{equation}
where $\bar{\tau}^\Lambda$ is the \emph{harmonic extension} of $\tau$ to $\Lambda$. This property, as well as its extension to the ground energy distribution, are discussed in Section~\ref{subsec:main}, in relation with the main identity presented there. Consequently, for many properties of interest, one may restrict attention to zero boundary values.

\subsection{Main results}\label{sec:main results}
We proceed to state the main results of this work, which concern the fluctuations of the height and energy of the minimal surface. The results apply to all codimensions $n\ge 1$ and all Hurst parameters $H\in(0,1)$.

Our surfaces are defined on
\begin{equation}
	\Lambda_L\coloneqq\{-L,\dots,L\}^d
\end{equation}
with zero boundary values. We write $\std(X)$ and $\med(X)$ for the standard deviation and median, respectively, of a random variable $X$. We write $|A|$ for the cardinality of a finite set~$A$.

\subsubsection{Low dimensions ($d<4$)}

Our first theorem shows that the typical heights and the minimal energy fluctuations of minimal surfaces in fractional Brownian disorder in dimensions $d<4$ are indeed described by power laws with the exponent formulas~\eqref{eq:exponent predictions}. 
Moreover, the theorem identifies the exponents governing the stretched exponential tail decay of the height and minimal energy distributions.

\begin{theorem}\label{thm:main 123} Let $d\in\{1,2,3\}$. There are $C,c>0$, depending only on the Hurst parameter $H$ and the codimension $n$, such that the following holds for all $t>0$ and integer $L\ge 1$:
\begin{enumerate}
    \item Height:
\begin{equation}\label{eq:123 height fluctuations thm}
    ce^{-Ct^{4-2H}} \le \mathbb P \Big( \max _{v\in \Lambda _L} \|\varphi ^{\eta,\Lambda_L} _v\| \ge tL^{\frac{4-d}{4-2H}} \Big) \le Ce^{-ct^{4-2H}}.
\end{equation}
    \item Minimal energy fluctuations:
    \begin{equation}\label{eq:bound height 123}
        c L^{\frac{4-d}{2-H}+d-2}\le \std(\GE^{\eta,\Lambda_L})\le C L^{\frac{4-d}{2-H}+d-2}.
    \end{equation}
    and, moreover,
    \begin{equation}\label{eq:lowerbound26}
        ce^{-Ct^{2-H}} \le \mathbb P \big(  |\GE ^{\eta,\Lambda_L} -\mathbb E [\GE^{\eta,\Lambda_L} ]| \ge tL^{\frac{4-d}{2-H} + d-2} \big) \le Ce^{-ct^{2-H}}.
    \end{equation}
\end{enumerate}
\end{theorem}

We make two remarks regarding the theorem.

First, the estimate~\eqref{eq:123 height fluctuations thm} gives a lower bound on the probability that the \emph{maximum} of the minimal surface delocalizes. In fact, we establish the stronger fact that a \emph{constant fraction} of the minimal surface delocalizes with the same probability bound. Precisely, in the notation of the theorem, we show that
\begin{equation}\label{eq:constant fraction delocalization}
    \mathbb P \left(  \big|\big\{v\in\Lambda_L:\|\varphi_v^{\eta,\Lambda_L}\|\ge tL^{\frac{4-d}{4-2H}} \big\} \big| \ge c |\Lambda _L| \right) \ge c e^{-Ct^{2-H}}.
\end{equation}

Second, the constants appearing in the lower bounds of the theorem can be chosen to be independent of $n$.

\smallskip
Lastly, just prior to the first appearance of this paper, Otto, Palmieri and Wagner published the work~\cite{otto2025minimizing} which studies the model~\eqref{eq:formal Hamiltonian} in the $1+1$-dimensional case with Brownian disorder. The main focus of~\cite{otto2025minimizing} is different from ours, presenting a study of the elastic and disorder energies of the minimal surface (rather than the minimal energy fluctuations) and discovering the appearance of logarithmic corrections to naive scaling. There is, however, some overlap with our results: a consequence of~\cite[Proposition 2]{otto2025minimizing} is that the \emph{maximal height} of the minimal surface has order at most $L$, and that the probability that it exceeds $tL$ is at most $C\exp(-c_s t^s)$ for any $s<3$. In addition, in~\cite[Corollary 2]{otto2025minimizing}, the probability that the height at a \emph{given vertex} exceeds $tL$ is shown to be at most $C\exp(-ct^3)$. These results partially overlap the $d=n=1$ and $H=1/2$ case of the upper bound in~\eqref{eq:123 height fluctuations thm}. 

\subsubsection{Critical dimension ($d=4$)}

Our next result establishes $d=4$ as the critical dimension, in which the minimal surface delocalizes, but only at a sub-power-law rate. 

\begin{theorem}\label{thm: main 4} Let $d=4$. There are $C,c>0$, depending only on the Hurst parameter $H$ and the codimension $n$, such that the following holds for all $t>0$ and integer $L\ge 2$:
\begin{enumerate}
    \item Height: On the one hand, for all $v\in \Lambda _L$,
    \begin{equation}\label{eq:bound h d=4}
     \mathbb P \big( \|\varphi ^{\eta,\Lambda_L} _v\| \ge t(\log L)^{\frac{5}{4-2H}} \big) \le Ce^{-ct^{4-2H}}
    \end{equation}
    and on the other hand,
    \begin{equation}\label{eq:fluctuation height d=4}
      \E\bigg[\frac{1}{|\Lambda_L|}\Big| \Big\{u\in\Lambda_L:\|\varphi_u^{\eta,\Lambda_L}\|\ge c (\log \log L)^{\frac{1}{4-2H}}\Big\} \Big| \bigg]\ge c.
    \end{equation}
    \item Minimal energy fluctuations:
    \begin{equation}\label{eq:27}
      c L^2(\log \log L)^{\frac{H}{4-2H}}\le \std(\GE^{\eta,\Lambda_L})\le C L^{2} (\log L)^{\frac{5H}{4-2H}}
    \end{equation}
    and, moreover,
    \begin{equation}\label{eq:boundGE4}
       \mathbb P \big(  |\GE ^{\eta,\Lambda_L} -\mathbb E [\GE^{\eta,\Lambda_L} ]| \ge tL^{2} (\log L)^{\frac{5H}{4-2H}} \big) \le Ce^{-ct^{2-H}}.
    \end{equation}
\end{enumerate}
\end{theorem}
Again, the constants in the lower bounds of these estimates can be chosen independent of~$n$.

\smallskip
The theorem does not capture the precise order of the typical height, showing only that it lies between $(\log\log L)^{\frac{1}{4-2H}}$ and $(\log L)^{\frac{5}{4-2H}}$. Nevertheless, we are able to show that a scaling relation still holds at the critical dimension, relating the typical height and minimal energy fluctuations. Specifically, we find that a version of the scaling relation~\eqref{eq:second scaling relation} holds at the critical dimension $d=4$, in the sense that if the typical height of the minimal surface is of order~$h$ then the typical fluctuations of the minimal energy are of order $L^2 h^H$. Let us state this precisely: We measure the fluctuations of the minimal energy using the function
\begin{equation}
    f(E):=\P\left(\left|\GE^{\eta,\Lambda_L} -\med(\GE^{\eta,\Lambda_L}) \right|\ge E\right).
\end{equation}
The typical height is measured using two notions: The first is a height that a fraction of the minimal surface reaches with uniformly positive probability, defined by
\begin{equation}
    h_- := \sup\left\{h\colon \mathbb E \big[ | \{v\in \Lambda_L : \|\varphi^{\eta,\Lambda_L}_v\| \ge h \} | \big] \ge \frac{1}{2}|\Lambda_L|\right\}.
\end{equation}
The second is a quantile of the normalized $\ell^{2H}$ norm of the minimal surface, defined by
\begin{equation}
    h_+ := \inf\Bigg\{h\colon \P\Bigg(\Big(\frac{1}{|\Lambda_L|}\sum _{v\in \Lambda_L} \|\varphi^{\eta,\Lambda_L}_v\|^{2H}\Big)^{\frac{1}{2H}}\ge h\Bigg)\le \frac{1}{3}\Bigg\}.
\end{equation}
It seems reasonable, but is not proved here, that these two heights are of comparable size. We show that the typical fluctuations of the minimal energy are at least of order $L^2 (h_-)^H$ and at most of order $L^2 (h_+)^H$. Precisely, for $C,c,\eps>0$ depending only on $H$, we have that
\begin{equation}\label{eq:second scaling relation for d=4}
    f(cL^2(h_-)^H) \ge \eps\quad\text{and}\quad
    f(CL^2(h_+)^H) \le 1-\eps.
\end{equation}
The lower bound is a consequence of Theorem~\ref{thm:GEAC}, while the upper bound follows from Proposition~\ref{as:conc}.

We remark that Theorem~\ref{thm:GEAC} and Proposition~\ref{as:conc} contain additional information and are also used in our analyses of other dimensions, where our results are more precise.

A second remark is that, with the above interpretation of the relations, it is not possible for both scaling relations,~\eqref{eq:scaling relation} and~\eqref{eq:second scaling relation}, to hold in dimension $d=4$, as the relation~\eqref{eq:scaling relation} would say that the minimal energy fluctuations relate to the typical height $h$ as $L^2 h^2$, while the relation~\eqref{eq:second scaling relation} would say that they relate as $L^2 h^H$ (and $h\to\infty$ as $L\to\infty$ by Theorem~\ref{thm: main 4}).

\subsubsection{High dimensions ($d>4$)}

Our final result shows that the minimal surfaces are localized in dimensions $d>4$, with minimal energy fluctuation exponent $\chi=\frac{d}{2}$. We further identify the exponent in the stretched exponential tail decay of the height distribution.

\begin{theorem}\label{thm:main d>4} Let $d\ge5$. There are $C,c>0$, depending only on the dimension $d$, the Hurst parameter $H$ and the codimension $n$, such that the following holds for all $t>0$ and integer $L\ge 2$:
\begin{enumerate}
    \item Height: for all $v\in \Lambda _L$,
    \begin{equation}\label{eq:loweboundheight d>4}
    ce^{-Ct^{4-2H}} \le \mathbb P \big( \|\varphi ^{\eta,\Lambda_L}_v\| \ge t \big) \le Ce^{-ct^{4-2H}}.
    \end{equation}
    \item Minimal energy fluctuations:
    \begin{equation}\label{eq:GE d>4}
     c L^{\frac{d}{2}}\le \std(\GE^{\eta,\Lambda_L})\le C L^{\frac{d}{2}}
    \end{equation}
    and, moreover,
    \begin{equation}\label{eq:boundGEd>4}
     \mathbb P \big(  |\GE ^{\eta,\Lambda_L} -\mathbb E [\GE^{\eta,\Lambda_L} ]| \ge tL^{d/2} \big) \le Ce^{-ct^{2-H}}.
    \end{equation}
\end{enumerate}
\end{theorem}
Again, the constants in the lower bounds of these estimates can be chosen independent of~$n$.

\subsection{Predictions from the physics literature}\label{sec:predictions}

Grinstein--Ma~\cite{grinstein1982roughening,grinstein1983surface} and Villain~\cite{villain1982commensurate} predicted in 1982 that for Brownian disorder (the case $H=1/2$ and $n=1$, also called random-field disorder), the minimal surface will delocalize with roughness exponent $\xi = \frac{4-d}{3}$ in dimensions $d<4$, delocalize to a sub-power-law height in dimension $d=4$ and be localized in dimensions $d>4$.

Later authors, including Kardar~\cite{kardar1987domain}, Nattermann~\cite{nattermann1987interface}, Halpin-Healy~\cite{halpin1989diverse, halpin1990disorder}, M\'ezard--Parisi~\cite{mezard1990interfaces, mezard1991replica, mezard1992manifolds} and Balents--Fisher~\cite{balents1993large} (see also the reviews~\cite{nattermann1988random, forgacs1991behavior, wiese2003functional, giamarchi2009disordered,wiese2022theory}) extended this prediction to general codimensions $n$ and disorders with power-law correlations. Their methods include scaling, functional renormalization group and replica symmetry breaking ideas. They typically consider the case when the disorder $\eta$ is centered Gaussian, white in space (i.e., in our discrete space, $(\eta_{v,\cdot})$ being independent between the vertices $v$) and, for each~$v$, has covariance function satisfying
\begin{equation}\label{eq:disorder correlations in physics}
    \Cov(\eta_{v,t},\eta_{v,s}) = f(\|t-s\|) \sim \|t-s\|^{-\gamma},
\end{equation}
for large values of $\|t-s\|$, with some $\gamma\ge -2$. For $\gamma>0$ one supposes that $f(r)$ is smooth for small values of $r$. The formula~\eqref{eq:disorder correlations in physics} does not seem to make literal sense for $\gamma<0$, but should perhaps be interpreted in this case as $\var(\eta_{v,t}-\eta_{v,s})\sim \|t-s\|^{-\gamma}$, in which case it can accommodate fractional Brownian disorder by choosing
\begin{equation}\label{eq:gamma to H}
    \gamma=-2H.
\end{equation}
We proceed with this interpretation in mind, and parametrize in terms of $H$, using~\eqref{eq:gamma to H}, also for $\gamma\ge 0$. In dimensions $d<4$ and all codimensions $n$, it is predicted that the exponent formulas~\eqref{eq:exponent predictions}, also termed \emph{Flory exponents}, hold for disorders with sufficiently strong correlations (the long-range regime), i.e., for $\gamma\le \gamma_c$ for some threshold $\gamma_c=\gamma_c(d,n)$. For disorders with weaker correlations, i.e., for $\gamma\ge\gamma_c$, the exponent values should fixate on the value that they take at $\gamma=\gamma_c$ (the short-range regime, which includes the focus of our earlier work~\cite{dembin2024minimal}).

The references above also agree that minimal surfaces with such disorder, for all values of~$\gamma$, localize in dimensions $d>4$.

Theorem~\eqref{thm:main 123} rigorously confirms that the Flory exponents~\eqref{eq:exponent predictions} describe the height and minimal energy fluctuations of the minimal surface in dimensions $d<4$ and all codimensions $n$ in the case $-2<\gamma<0$, when this case is realized by a fractional Brownian disorder of Hurst parameter $H\in (0,1)$ (the case $\gamma=-2$ is also known to hold, as the linear disorder model is exactly solvable). The result implies that $\gamma_c(d,n)\ge 0$ for all $d<4$ and $n\ge 1$.

Theorem~\ref{thm:main d>4} rigorously confirms the localization of the surfaces in dimensions $d>4$ and all codimensions $n$ in the case $-2<\gamma<0$ (realized by a fractional Brownian disorder).

\smallskip
At the critical dimension $d=4$, Emig--Natterman (\cite{emig1998roughening} and ~\cite[following (78)]{emig1999disorder}) predict that minimal surfaces with Brownian disorder ($H=1/2$ and $n=1$) will delocalizae to height $(\log L)^{1/3}$. More generally, a common idea in field theory is that if a surface delocalizes to height $L^{(\alpha+o(1)) \eps}$, as $\eps\to0$, in dimensions $d = d_c - \eps$, then it will delocalize to height of order $(\log L)^\alpha$ at the critical dimension $d=d_c$. In our setup, in light of the exponent formulas~\eqref{eq:exponent predictions}, this predicts that the minimal surfaces with fractional Brownian disorder will delocalize to height $(\log L)^{\frac{1}{4-2H}}$ at the critical dimension $d=4$, independently of $n$. We are not aware of a prediction in the physics literature for the minimal energy fluctuations in dimension $d=4$.

Theorem~\ref{thm: main 4} rigorously confirms the delocalization of the minimal surface in dimension $d=4$ and all codimensions $n\ge 1$, to a height of sub-power-law order, for fractional Brownian disorder of Hurst parameter $0<H<1$. It does not, however, capture the predicted order $(\log L)^{\frac{1}{4-2H}}$ (known to hold in the linear disorder case~$H=1$). Our results further find that a scaling relation holds also at the critical dimension $d=4$, which leads to the, possibly new, prediction that the minimal energy fluctuations in dimension $d=4$ are of order $L^2(\log L)^{\frac{H}{4-2H}}$.

\subsection{Reader's guide}
In Section \ref{sec:main identity and consequences}, we establish the main identity and derive several consequences that will be crucial for later proofs. The proofs of the main theorems (Theorems \ref{thm:main 123}, \ref{thm: main 4}, and \ref{thm:main d>4}) are spread across two sections, as the methods for proving the lower and upper bounds differ. We have grouped the proofs according to the techniques used. In Section \ref{sec:loc}, we prove the upper bounds on height fluctuations (the upper bounds in \eqref{eq:123 height fluctuations thm}, \eqref{eq:bound h d=4}, and \eqref{eq:loweboundheight d>4}) and derive concentration estimates for the ground energy (see the upper bounds \eqref{eq:lowerbound26}, \eqref{eq:boundGE4}, and \eqref{eq:boundGEd>4}). In Section \ref{sec:deloc}, we prove the lower bounds on the height fluctuations (the lower bounds in~\eqref{eq:123 height fluctuations thm}, \eqref{eq:fluctuation height d=4}, and \eqref{eq:loweboundheight d>4}) and derive anti-concentration estimates for the ground energy (see the lower bounds \eqref{eq:lowerbound26}, \eqref{eq:27}, and \eqref{eq:GE d>4}). The remaining bounds follow directly from those established above. Finally, in Section \ref{sec:assumptions for eta white}, we prove the existence and uniqueness of a minimizer, as well as provide crude a priori bounds for the concentration of the ground energy (that will be used in Section \ref{sec:loc}).

\subsection{Acknowledgements}

We thank Felix Otto and Christian Wagner for stimulating discussions, at Princeton University in Fall 2023 and at a workshop in Budapest in Winter 2025.

The research of R.P. is partially supported by the Israel Science Foundation grants
1971/19 and 2340/23, and by the European Research Council Consolidator grant 101002733 (Transitions).

Part of this work was completed while R.P. was a Cynthia and Robert Hillas Founders' Circle Member of the Institute for Advanced Study and a visiting fellow at the Mathematics Department of Princeton University. R.P. is grateful for their support.

\section{Main identity and first consequences}\label{sec:main identity and consequences}
This section introduces the main identity, which is the special feature of the model~\eqref{eq:formal Hamiltonian} on which our analysis is based. We start with required notation, proceed to describe the identity and present some of its first consequences.

\subsection{Notation}
\subsubsection*{Shifts} Given a function $s:\Z^d\to\R^n$ we let
\begin{equation}
    \supp(s) := \{v\in\Z^d\colon s_v\neq 0\}
\end{equation}
be the \emph{support} of $s$. 

\subsubsection*{Inner products and Laplacian} For $x,y\in\R^n$ we use the standard notation
    \begin{equation}\label{eq:inner product R n}
        x\cdot y:=\sum_{i=1}^n x_i y_i
    \end{equation}
    so that $\|x\|^2=x\cdot x$. We use the notation $(\cdot ,\cdot )$ for inner products of functions on the vertices:  
    Given $\varphi,\psi:\Z^d\to\R^n$ and $\Lambda\subset\Z^d$, we let
    \begin{equation}
        (\varphi,\psi)_\Lambda:=\sum_{v\in\Lambda} \varphi_v \cdot  \psi_v
    \end{equation}
    and also set $\|\varphi\|_\Lambda^2:=(\varphi,\varphi)_\Lambda$.
    In addition, we write
    \begin{equation}
        (\nabla \varphi, \nabla\psi)_\Lambda:= \sum_{\substack{u\sim v\\\{u,v\}\cap\Lambda\neq\emptyset}}(\varphi_u - \varphi_v)\cdot ( \psi_u - \psi_v)
    \end{equation}
    and also set $\|\nabla\varphi\|_\Lambda^2:=(\nabla\varphi,\nabla\varphi)_\Lambda$ (the Dirichlet energy of $\varphi$ on $\Lambda$). We use the abbreviations $(\varphi,\psi):=(\varphi,\psi)_{\Z^d}$ and $(\nabla \varphi,\nabla\psi):=(\nabla \varphi,\nabla\psi)_{\Z^d}$.

    We note the following useful \emph{discrete Green's identity},
    \begin{equation}\label{eq:Green's identity}
        (\nabla\varphi,\nabla\psi)_\Lambda=(\varphi,-\Delta_\Lambda\psi)=(-\Delta_\Lambda\varphi,\psi),
    \end{equation}
    which holds whenever one of the sums (and then all others) converges absolutely. Here,
    $\Delta_\Lambda$ is the discrete Laplacian which acts on $\varphi:\Z^d\to\R^n$ by
    \begin{equation}
        (\Delta_\Lambda\varphi)_v:=\sum_{u\,:\, \substack{u\sim v,\\ \{u,v\}\cap\Lambda\neq\emptyset}} (\varphi_u - \varphi_v).
    \end{equation}
    In particular, letting
    \begin{equation}
        \Lambda^+:=\{v\in\Z^d\colon \exists u, u\sim v\text{ and }\{u,v\}\cap\Lambda\neq \emptyset\}
    \end{equation}
    we have $\Delta_\Lambda\varphi=0$ off $\Lambda^+$ with this definition, so that, e.g., $(-\Delta_\Lambda\varphi,\psi) = (-\Delta_\Lambda\varphi,\psi)_{\Lambda^+}$. Again, we abbreviate $\Delta:=\Delta_{\Z^d}$.

    \subsubsection*{Harmonic extension and Dirichlet energy} Given a finite $\Lambda\subset\Z^d$ and $\tau:\Z^d\to\R^n$, let $\overline{\tau}^\Lambda$ be the harmonic extension of $\tau$ to $\Lambda$. Precisely, $\overline{\tau}^\Lambda:\Z^d\to\R^n$ is the unique function satisfying
    \begin{equation}
    \begin{split}
        &\overline{\tau}^\Lambda = \tau\quad\text{on $\Z^d\setminus\Lambda$},\\
        &\Delta_\Lambda(\overline{\tau}^\Lambda) = 0\quad\text{on $\Lambda$}.
    \end{split}
    \end{equation}
    This is equivalent, as is well known, to setting $\bar{\tau}^\Lambda_v = \E(\tau_{X_{T_\Lambda}})$ with $(X_t)_{t=0}^\infty$ a simple random walk on $\Z^d$ started at $v$ and $T_\Lambda:=\min\{t\ge 0\colon X_t\notin\Lambda\}$. In addition, we let
    \begin{equation}
    \|\tau\|_{\DE(\Lambda)}^2:=\|\nabla\bar{\tau}^\Lambda\|_\Lambda^2
    \end{equation}
    be the Dirichlet energy of the harmonic extension of $\tau$ to $\Lambda$.

\subsection{Main identity}\label{subsec:main} The starting point for our analysis of the disordered random surface model~\eqref{eq:formal Hamiltonian} is the following deterministic identity, which controls the change in the energy of surfaces under a shift of both the disorder (in the sense of~\eqref{eq:eta s def}) and the surface.

Our earlier work also relied on a deterministic identity~\cite[Proposition 2.1]{dembin2024minimal}. The identity here takes a more complicated form due to the more complicated definition~\eqref{eq:eta s def} of the shifted and recentered disorder $\eta^s$. The latter is required in order to work with disorders which only have stationary increments (rather than being stationary).

\begin{prop}[Main identity]\label{prop:main identity} Let $\eta:\Z^d\times\R^n\to(-\infty,\infty]$ and finite $\Lambda\subset\Z^d$. For each $s:\Z^d\to\R^n$ and $\varphi:\Z^d\to\R^n$ we have
\begin{equation}\label{eq:main identity}
    H^{\eta^s,\Lambda}(\varphi+s) - H^{\eta,\Lambda}(\varphi) = (\varphi,-\Delta_\Lambda s) + \frac{1}{2}\|\nabla s\|_\Lambda^2-\sum_{v\in\Lambda}\eta_{v,-s_v}.
\end{equation}
\end{prop}

\begin{proof}
We first observe that, by the definition~\eqref{eq:eta s def} of $\eta^s$, for each $v\in\Z^d$,
\begin{equation}
     \eta^s_{v,\varphi_v+s_v} = \eta_{v,\varphi_v}-\eta_{v,-s_v}.
\end{equation}
Therefore, using the discrete Green's identity~\eqref{eq:Green's identity},
\begin{equation}
\begin{split}
    H^{\eta^s,\Lambda}(\varphi+s) - H^{\eta,\Lambda}(\varphi) &= \frac{1}{2}(\|\nabla (\varphi+s)\|_\Lambda^2 - \|\nabla\varphi\|_\Lambda^2) +  \sum_{v\in\Lambda} (\eta^s_{v,\varphi_v+s_v}-\eta_{v,\varphi_v})\\
    &=\frac{1}{2}\big((\nabla(\varphi+s),\nabla(\varphi+s))_\Lambda - (\nabla\varphi, \nabla\varphi)_\Lambda\big)-\sum_{v\in\Lambda}\eta_{v,-s_v}\\
    &=(\nabla\varphi,\nabla s)_\Lambda + \frac{1}{2}(\nabla s, \nabla s)_\Lambda-\sum_{v\in\Lambda}\eta_{v,-s_v}\\
    &=(\varphi,-\Delta_\Lambda s) + \frac{1}{2}\|\nabla s\|_\Lambda^2-\sum_{v\in\Lambda}\eta_{v,-s_v}.\qedhere
\end{split}    
\end{equation}
\end{proof}

We proceed to discuss the effect of boundary values.

\begin{prop}[Effect of boundary values]\label{prop:effect of boundary conditions} Let $\Lambda\subset \Z^d$ finite and $\tau:\Z^d \to \R^n$. Then
\begin{alignat}{1}
\varphi^ {\eta,\Lambda,\tau }&=\varphi^ {\eta^ {-\bar{\tau}^\Lambda},\Lambda}+\bar{\tau}^\Lambda,\\
\GE^{\eta,\Lambda,\tau}&= \GE^ {\eta^ {-\bar{\tau}^\Lambda},\Lambda }+\frac{1}{2}\|\tau\|_{\DE(\Lambda)}^2-\sum_{v\in\Lambda}\eta^{-\bar{\tau}^\Lambda}_{v,-\bar{\tau}^\Lambda_v}.
\end{alignat}
Consequently, we have the identity in distribution,
\begin{equation}
    (\varphi^ {\eta,\Lambda,\tau}, \GE^{\eta,\Lambda,\tau})\eqd(\varphi^ {\eta,\Lambda}, \GE^{\eta,\Lambda}-\sum_{v\in\Lambda}\eta_{v,-\bar{\tau}^\Lambda_v}) + \Big(\bar{\tau}^\Lambda,\frac{1}{2}\|\tau\|_{\DE(\Lambda)}^2\Big).
\end{equation}
\end{prop}

\begin{proof}
    First note that $\varphi+{\bar{\tau}^\Lambda}\in \Omega^ {\Lambda,\tau}$ if and only if $\varphi\in\Omega^ {\Lambda}$.
        Second, Proposition \ref{prop:main identity} (applied to $\eta^{-\bar{\tau}^\Lambda}$ and $\bar{\tau}^\Lambda$) implies that for each $\varphi\in\Omega^ {\Lambda}$,
        \begin{equation*}
H^{\eta,\Lambda}(\varphi+{\bar{\tau}^\Lambda}) - H^{\eta^{-\bar{\tau}^\Lambda},\Lambda}(\varphi) = (\varphi,-\Delta_\Lambda {\bar{\tau}^\Lambda}) + \frac{1}{2}\|\nabla \bar{\tau}^\Lambda\|_\Lambda^2-\sum_{v\in\Lambda}\eta^{-\bar{\tau}^\Lambda_v}_{v,-\bar{\tau}^\Lambda_v}= \frac{1}{2}\|\tau\|_{\DE(\Lambda)}^2-\sum_{v\in\Lambda}\eta^{-\bar{\tau}^\Lambda_v}_{v,-\bar{\tau}^\Lambda_v}
\end{equation*}
using that $\Delta_\Lambda(\overline{\tau}^\Lambda) = 0$ on $\Lambda$ and $\varphi=0$ off $\Lambda$. Thus,
\begin{equation}\label{boundary}
    \text{$\varphi$ minimizes $H^{\eta^{-\bar{\tau}^\Lambda},\Lambda}$ over $\Omega^ {\Lambda}$ if and only if $\varphi+{\bar{\tau}^\Lambda}$ minimizes $H^{\eta,\Lambda}$ over $\Omega^ {\Lambda,\tau}$}.
\end{equation}
The second part of the conclusion, follows from the fact that for any shift function $s$, the environments $\eta^s$ and $\eta$ are identically distributed.
\end{proof}

\subsection{Concentration estimates}

Recall that the notation $\eta_A$ refers to $(\eta_{v,\cdot})_{v\in A}$ (as defined before assumption~\ref{as:conc}). The following concentration estimate is proved in Section \ref{sec:assumptions for eta white} using general concentration inequalities for Gaussian processes (stemming from the fundamental work of Borell and Tsirelson--Ibragimov--Sudakov). 

\begin{prop}[Concentration for the ground energy]\label{as:conc} 
There exists a universal constant $c>0$ such that for every finite  $\Delta\subset\Lambda \subset \mathbb Z ^d$, every $\tau:\Z^d\to\R^n$ and every $r,h>0$ we have almost surely
\begin{equation*}
\begin{split}
    \inf _{\gamma \in \mathbb R} \P \big( \big|\GE^{\eta,\Lambda,\tau} - \gamma  \big| \ge r \mid \eta_{\Delta^c} \big)&\le 2 \exp \Big(- \frac{cr^2}{h^{2H}|\Delta|} \Big) +\mathbb P \Big(  \frac{1}{|\Delta |}\sum _{v\in \Delta } \|\varphi^{\eta,\Lambda,\tau}_v\|^{2H} \ge h^{2H} \mid \eta_{\Delta^c}\Big)\\
    &\le 2 \exp \Big(- \frac{cr^2}{h^{2H}|\Delta|} \Big) +\mathbb P \Big(  \max _{v\in \Delta } \|\varphi ^{\eta,\Lambda,\tau }_v\| \ge h \mid \eta_{\Delta^c}\Big).
\end{split}
\end{equation*}
\end{prop}

Combining the last proposition with the main identity we obtain the following estimate on the minimal surface.

\begin{lemma}[Concentration for linear functionals of the minimal surface]\label{lem:fluctuation and concentration} There are universal constants $C,c>0$ such that for every finite $\Lambda\subset\Z^d$, a function $s:\Z^d\to\R^n$ and $h\ge r>0$ we have
\begin{equation}
\P\left(\left|(\varphi^{\eta, \Lambda},-\Delta_\Lambda s)\right|\ge r \right)  \le C\exp \Big( -\frac{cr^4  }{h^{2H}D(s)}  \Big) +3\mathbb P \Big( \max _{v\in \supp (s) } \| \varphi ^{\eta,\Lambda}_v \| \ge h \Big),
\end{equation}
where $D(s)$ is given by 
\begin{equation}
    D(s):= \|\nabla s\|_\Lambda ^{4-4H} \|s\|_{2H}^{2H} + \|\nabla s\|_\Lambda ^4 |\supp (s)|
\end{equation}
and $\|s\|_{2H}^{2H}:=\sum _{v\in \Lambda } \|s_v\|^{2H}$.
\end{lemma}

\begin{proof}
    For brevity, we write $H^{\eta}$, $\varphi^\eta$ and $\GE^\eta$ for $H^{\eta,\Lambda}$, $\varphi^{\eta,\Lambda}$ and $\GE^{\eta,\Lambda}$, respectively.
    
    By the main identity (Proposition~\ref{prop:main identity}), for each $\rho\in\R$, 
    \begin{equation}
       \GE ^{\eta ^{\rho s}}-\GE ^\eta  \le H^{\eta^{\rho s}}(\varphi^{\eta}+\rho s) - H^{\eta}(\varphi^{\eta}) = \rho(\varphi^{\eta},-\Delta_\Lambda s) + \frac{\rho^2}{2}\|\nabla s\|_\Lambda^2- \sum _{v\in \Lambda } \eta _{v,-\rho s_v}.
    \end{equation}
    It follows that for $\rho=  -\frac{r}{\|\nabla s\|_\Lambda^2}$ we have the containment
    \begin{equation}
        \left\{(\varphi^{\eta}, -\Delta_\Lambda s)\ge r\right\} \subset\left\{\GE ^{\eta ^{\rho s}}-\GE ^\eta + \sum _{v\in \Lambda } \eta _{v,-\rho s_v } \le -\frac{r^2}{2\|\nabla s\|_\Lambda ^2} \right\} .
    \end{equation}
    Next, define the event  $\mathcal A ^-:= \{\sum _{v\in \Lambda } \eta _{v,-\rho s_v } \le -r^2/(4\|\nabla s\|_\Lambda ^2)\}$. We obtain for any $\gamma\in\R$,
    \begin{equation}\label{eq:bound for +r}
    \left\{(\varphi^{\eta}, -\Delta_\Lambda s)\ge r\right\} \subset\left\{|\GE^{\eta^{\rho s}} - \gamma|\ge \frac{r^2}{8\|\nabla s\|_\Lambda ^2}\right\}\cup\left\{|\GE^{\eta} - \gamma|\ge \frac{r^2}{8\|\nabla s\|_\Lambda ^2}\right\}\cup \mathcal A ^-. 
    \end{equation}
    Repeating the above derivation with $r$ replaced by $-r$, we also conclude that, for any $\gamma\in\R$,
    \begin{equation}\label{eq:bound for -r}
    \left\{(\varphi^{\eta}, -\Delta_\Lambda s)\le -r\right\} \subset\left\{|\GE^{\eta^{-\rho s}} - \gamma|\ge \frac{r^2}{8\|\nabla s\|_\Lambda^2}\right\}\cup\left\{|\GE^{\eta} - \gamma|\ge \frac{r^2}{8\|\nabla s\|_\Lambda^2}\right\} \cup \mathcal A ^+,
    \end{equation}
    where $\mathcal A ^+:= \{\sum _{v\in \Lambda } \eta _{v,\rho s_v } \ge r^2/(4\|\nabla s\|_\Lambda ^2)\}$. Next, let $\mathcal F$ be the sigma algebra generated by the noise $\eta _v$ for all $v\notin \supp (s)$. Conditionally on $\mathcal F $, the ground energies $\GE^{\eta^{\rho s}}$, $\GE^{\eta^{-\rho s}}$ and $\GE^\eta$ all have the same distribution. Moreover, $\mathcal A ^+$ and $\mathcal A ^{-}$ are independent of $\mathcal F $ and have the same probability. Using these facts together with \eqref{eq:bound for +r} and \eqref{eq:bound for -r} we obtain
    \begin{equation}\label{eq:773}
        \P\left(\left|(\varphi^{\eta },-\Delta_\Lambda s)\right|\ge r\mid \mathcal F \right) \le 3\inf_{\gamma\in\R}\P\left(\left|\GE^{\eta }-\gamma\right|\ge \frac{r^2}{4\|\nabla s\|_\Lambda^2}\mid  \mathcal F \right) +2 \mathbb P (\mathcal A ^+).
    \end{equation}
     By Proposition~\ref{as:conc} with $\Delta =\supp (s)$ we have
    \begin{equation*}      \inf_{\gamma\in\R}\P\left(\left|\GE^{\eta }-\gamma\right|\ge \frac{r^2}{4\|\nabla s\|_\Lambda^2}\mid  \mathcal F \right) \le C \exp \Big(- \frac{cr^4}{h^{2H} \|\nabla s\|_\Lambda^4|\supp (s)|} \Big) +\mathbb P \big(  \max _{v\in \supp (s) } \|\varphi ^\eta  _v\| \ge h \mid \mathcal F \big)
    \end{equation*}
    Next, we have that  $\sum _v \eta _{v,\rho s_v} \sim N(0,\sigma ^2 )$ where $\sigma ^2 =\sum _v \| \rho s_v\|^{2H} =r^{2H}\|\nabla s\|_\Lambda^{-4H} \|s\|_{2H}^{2H} $  and therefore 
    \begin{equation*}
        \mathbb P ( \mathcal A^+  )\le C\exp \Big(- \frac{cr^{4-2H}}{\|\nabla s\|_\Lambda ^{4-4H} \|s\|_{2H}^{2H} } \Big)
    \end{equation*}
   Substituting these bounds into \eqref{eq:773} and taking expectations on both sides finishes the proof of the lemma.
\end{proof}

\section{Localization}\label{sec:loc}

In this section, we prove upper bounds for the height and ground energy fluctuations from Theorems \ref{thm:main 123}, \ref{thm: main 4} and \ref{thm:main d>4}. 
\subsection{Estimates on Green's function}
Throughout this section we often use the Green's function $G_\Lambda ^v$ which is defined as follows. Let $\Lambda\subset \mathbb Z^d$. Let $\partial\Lambda$ denote the exterior boundary of $\Lambda$, that is $\partial \Lambda:=\Lambda^+\setminus \Lambda$ where $\Lambda ^+=\{v\in \mathbb Z ^d : \exists u\sim v \text{ with } \{u,v\}\cap \Lambda \neq \emptyset \}$. The Green's function in the domain $\Lambda $ for $v\in\Lambda$ is defined by 
\[\forall x\in\Z^ d\qquad G_\Lambda^v(x):=\frac{1}{2d} \cdot \E_x\left[ \big| \big\{ t\in [0, \tau _{\Lambda }] : X_t=v\big\} \big| \right],\]
where $(X_t)_{t\ge0}$ is a simple discrete-time random walk on $\Z^d$ with $X_0=x$ and $\tau _\Lambda $ is the first exit time of $\Lambda $,  $\tau_{\Lambda }:=\min\{t\ge 0:X_t\notin \Lambda \}$. When $\Lambda =\Lambda _L$ we write $G_L ^v:=G_{\Lambda _L}^v$. Note that $\Delta G_\Lambda ^v (v)=-1$ and $\Delta G_\Lambda ^v (u)=0$ for any  $u\notin \partial \Lambda \cup \{v\}$. 

We will need the following estimates for the Green's function in a box. See \cite[Lemma~3.6]{dembin2024minimal} for the proof.

\begin{lem}\label{lem:Green}
    Let $\Lambda :=[a_1,b_1]\times \cdots \times [a_d,b_d]$ be a box. For $v\in \Lambda $ we write $r_v:=d(v,\Lambda ^c)$. We have:    \begin{enumerate}
        \item 
        For all $v,x\in \Lambda $ such that $r_v \le 2\|v-x\|$ we have 
        \begin{equation}
           G_\Lambda ^v(x)\le Cr_xr_v \|x-v\|^{-d}. 
        \end{equation}
        \item 
        For all $v,u,x\in \Lambda $ such that $r_v\le 2\|x-v\|$ we have 
        \begin{equation}
           |G_\Lambda ^v(x)-G_\Lambda ^u(x) |\le Cr_x\|u-v\|\cdot \|x-v\|^{-d}. 
        \end{equation}
    \end{enumerate}
\end{lem}

We will also need the following result estimating differences of Green's functions with respect to two boxes.

\begin{cor}\label{cor:green}
    Let $\Lambda :=[a_1,b_1]\times \cdots \times [a_d,b_d]$ and let $v\in \Lambda $. Let $r\ge d(v,\Lambda ^c )/2$ and let $\Delta := \Lambda \cap (v+(-r,r)^d)$. Then: 
    \begin{enumerate}
        \item 
         For any $x\in \Lambda $ we have 
    \begin{equation}
        |G_\Lambda ^v (x)-G_\Delta ^v (x) | \le Cd(v,\Lambda ^c) r^{1-d}.
    \end{equation}
        \item 
    For any $u\in \Delta $ and $x\in \Lambda $ we have
        \begin{equation}
        | \big( G_\Lambda ^v(x)-G_\Lambda ^u(x) \big) -\big(  G_\Delta ^v(x)-G_\Delta ^u(x) \big) | \le C \|u-v\| r^{1-d}.
    \end{equation}
    \end{enumerate}
\end{cor}

\begin{proof}
We start with the first part. We may couple the walks in the definition of $G_\Lambda ^v(x)$ and $G_\Delta ^v(x)$ to be the same walk $X_t$. This gives 
\begin{equation}
\begin{split}
    G_\Lambda ^v(x)-G_\Delta ^v(x)= \frac{1}{2d} \cdot \E_x\left[ \big| \big\{ t\in (\tau _\Delta , \tau _{\Lambda }] : X_t=v\big\} \big| \right]=\frac{1}{2d} \cdot  \mathbb E _x[G_\Lambda ^v(X(\tau _\Delta )) ] \\
    \le C d(v,\Lambda ^c)  \mathbb E _x[  \|X(\tau _{\Delta })-v\|^{1-d} ] \le Cd(v,\Lambda ^c) r^{1-d},
\end{split}
\end{equation}
where in the last inequality we used that $X(\tau _{\Delta }) \notin \Delta $ and therefore $\|X(\tau _{\Delta})-v\| \ge r$, and in the second to last inequality we used Lemma~\ref{lem:Green} and the fact that $d(X(\tau _{\Delta }),\Lambda ^c) \le \|X(\tau _{\Delta})-v\|+d(v,\Lambda ^c) \le  \|X(\tau _{\Delta})-v\|+2r \le 3\|X(\tau _{\Delta})-v\|$. Note that in this proof if $x\in \Lambda \setminus \Delta $ then $X(\tau _\Delta )=x$ and all the bounds stated still hold.  This finishes the proof of the first part.

For the second part, we similarly have that
\begin{equation*}
    \big( G_\Lambda ^v(x)-G_\Lambda ^u(x) \big) -\big(  G_\Delta ^v(x)-G_\Delta ^u(x) \big)= \frac{1}{2d}\cdot  \mathbb E _x  \big[ G_\Lambda ^v(X(\tau _\Delta )) -G_\Lambda ^u(X(\tau _\Delta ))  \big]\le C\|u-v\| r^{1-d},
\end{equation*}
where the last inequality is by the second part of Lemma~\ref{lem:Green}.
\end{proof}
\subsection{Upper bound on height fluctuations}
\subsubsection{Dimensions $d\ge 5$}

In this section we prove the upper bound in \eqref{eq:loweboundheight d>4} from Theorem \ref{thm:main d>4}. Define
\begin{equation}
    q(h):= \max _{v\in \Lambda } \mathbb P \big( \|\varphi ^\eta _v \| \ge h \big).
\end{equation}

Our goal will be to prove the following recursive bound on the function $q$.
\begin{lem}\label{lem:loc2}
    Suppose that $d\ge 5$. Then, for any $r\ge 1$ we have 
    \begin{equation}
        q(r) \le C\exp(-cr^{4-2H})  +C\sum _{j=1}^{\lceil \log _2 L\rceil } 2^{jd} q(2^{j/8}r)  
    \end{equation}
\end{lem}

For the proof of Lemma~\ref{lem:loc2}, we fix $v\in \Lambda _L$ and a unit vector $e\in \mathbb R ^n$, and bound the fluctuations of $| \varphi ^\eta _v \cdot e |$. We do this by analyzing separately the contribution from each dyadic scale. To this end, define 
\begin{equation*}
\Lambda _j:=\Lambda_L  \cap (v+(-2^j,2^j)^d)
\end{equation*}
and let $m$ be the first integer for which $\Lambda _m=\Lambda _L$ (so that $m\approx \log _2L$). For $1\le j\le m$ we consider the Green's functions $G_j^v(x):=G_{\Lambda _j}^v(x)$. Finally, for $2\le j \le m$ define the shift functions 
\begin{equation}\label{def:s_j}
s_j(x)=s_j^v(x):=(G_j^v(x)-G_{j-1}^v(x) )e    
\end{equation}
and for $j=1$ define  $s_1(x)=s_1^v(x):=G_1^v(x)e$. Recall the function $D(s)$ of a shift $s$ defined in Lemma~\ref{lem:fluctuation and concentration}.  We start with the following claim which bounds $D(s_j)$ and will be useful in any dimension.

\begin{claim}\label{claim:D(s)}
    Let $d\ge 1$ and $v\in\Lambda _L$. Set $r_v:=d(v,\Lambda _L^c)$. For all $1\le j \le m$ we have 
    \begin{equation}
        D(s_j^v) \le C \min \big( 2^{j(4-d)} ,r_v^{4-2H} 2^{j(2H-d)}\big).
    \end{equation}
\end{claim}

\begin{proof}
    Let $m':=\lfloor \log _2r_v \rfloor -1$.  It suffices to prove that 
    \begin{equation}\label{eq:sj}
        \|s_j\|_{2H}^{2H} \le C\begin{cases} 2^{dj+2H(2-d)j} & \ \, 1\le j\le m' \\ r_v^{2H}2^{dj+2H(1-d)j} & m'<j \le m\end{cases}   , \quad  \|\nabla s_j\|_\Lambda ^2 \le C\begin{cases}2^{j(2-d)} & \ \, 1\le j\le m' \\ r_v^22^{-dj} & m'<j \le m\end{cases}. 
    \end{equation}
    Indeed, we have that  $|\supp (s_j)| \le C2^{jd}$ and therefore for $m'<j\le m$
        \begin{equation}
        \begin{split}
             D(s_j)& =\|\nabla s _j\|_\Lambda ^{4-4H} \|s_j\|_{2H}^{2H} + \|\nabla s_j\|_\Lambda ^4 |\supp (s_j)| \\
             &\le C( r_v^22^{-dj})^{2-2H}r_v^{2H}2^{dj+2H(1-d)j}+C(r_v^22^{-dj})^22^{jd} \\
             &\le Cr_v^{4-2H} 2^{j(2H-d)}+Cr_v^42^{-jd}\le  Cr_v^{4-2H} 2^{j(2H-d)}
        \end{split}
    \end{equation}
    where we used in the last inequality that $r_v\le 2^{m'+1}\le 2^j$. The computation for $1\le j\le m'$ follows by replacing $r_v$ by $2^j $ in the previous computation. Hence, 
    \begin{equation}
        D(s_j)  \le  C  \min \big( 2^{j(4-d)} ,r_v^{4-2H} 2^{j(2H-d)}\big).
    \end{equation}
    We turn to prove \eqref{eq:sj}. By the first part of Corollary~\ref{cor:green} we have $\max _{x}\|s_j(x)\|\le C2^{j(2-d)}$ when $j\le m'$ and $\max _{x}\|s_j(x)\|\le Cr_v2^{j(1-d)}$ when $m'<j\le m$ therefore the first estimate in \eqref{eq:sj} follows (recall that $\|s\|_{2H}^{2H}:=\sum _{v\in \Lambda } \|s_v\|^{2H}$).
    
    Next, we estimate the Dirichlet energy of $s_j$. For all $1<j\le m$ we have that 
    \begin{equation}\label{eq:laplacian norm}
        \|\nabla s_j\|^2_{\Lambda } = (s_j,-\Delta s_j)=\big( G_j^v-G_{j-1}^v,\Delta (G_j^v-G_{j-1}^v) \big)=-\!\!\! \sum _{x\in \partial \Lambda _{j-1}\setminus\partial \Lambda } \! G_j^v(x) \Delta G_{j-1}^v(x),
    \end{equation}
    where in the last equality we used that $\Delta (G_j^v-G_{j-1}^v)$ is supported on $\partial \Lambda _j \cup \partial \Lambda _{j-1}$ and that $G_j$ is supported on $\Lambda _j$ and $G_{j-1}$ is supported on $\Lambda _{j-1} $.  We would like to use Lemma~\ref{lem:Green} in order to bound the right hand side of \eqref{eq:laplacian norm} and to consider separately the cases $1<j\le m'$ and $m'<j \le m$. When $1<j\le m'$, by Lemma~\ref{lem:Green}, for all $x\in \partial \Lambda _{j-1} $ we have $|G_j^v(x)| \le C2^{j(2-d)}$ and $|\Delta G_{j-1}(x)| = | G_{j-1}(y)|\le C2^{j(1-d)}$, where $y\sim x$ is the unique neighbour of $x$ in $\Lambda _{j-1}$. Similarly, when $m'<j\le m$, for all $x\in \partial \Lambda _{j-1}\setminus \partial \Lambda $ we have $|G_j^v(x)| \le Cr_v2^{j(1-d)}$ and $|\Delta G_{j-1}(x)| = | G_{j-1}(y)|\le Cr_v2^{-dj}$. Substituting these estimates into \eqref{eq:laplacian norm} finishes the proof of the second estimate in \eqref{eq:sj} when $1<j\le m$. This estimate clearly holds in the case $j=1$. 
\end{proof}

\begin{proof}[Proof of Lemma~\ref{lem:loc2}]
Fix $v\in \Lambda _L$ and recall the definition of $s_j=s_j^v$ in \eqref{def:s_j}. We have that $G_m^ve=\sum _{j=1}^ms_j^v$ and therefore
    \begin{equation}\label{eq:phi i}
        \varphi ^{\eta } _v \cdot e=(\varphi ^\eta ,-\Delta G_m^ve)= \sum _{j=1}^m (\varphi^ \eta ,- \Delta s_j).
    \end{equation}
Next, we bound the fluctuations of each of the terms on the right hand side of \eqref{eq:phi i}. To this end, let $c_0$ be a constant so that $c_0\sum _{j=1}^{\infty }d2^{-j/8}<1$. By Claim~\ref{claim:D(s)} we have that $D(s_j)\le C2^{j(4-d)}\le C2^{-j}$ and therefore by Lemma~\ref{lem:fluctuation and concentration} with $r$ being $c_02^{-j/8}r$ and  $h$ being $2^{j/8}r$ we have
\begin{equation}
\begin{split}
 \P\left(\left|(\varphi^{\eta},-\Delta s_j)\right|\ge c_02^{-j/8}r \right) &\le C\exp \big(-cr^{4-2H} 2^{j/4} \big)+ 3\mathbb P \big(  \max _{v\in \Lambda _j} \|\varphi ^\eta _v\| \ge 2^{j/8}r  \big) \\
 &\le C\exp \big(-cr^{4-2H} 2^{j/4} \big)+ C2^{dj}q(2^{j/8}r),
\end{split}
\end{equation}
where in the first inequality we also used that $H\le 1$. Hence, by the choice $c_0$ we have 
\begin{equation}
\begin{split}
    \mathbb P \big( | \varphi ^\eta  _v \cdot e | \ge r/d \big) &\le \sum _{j=1}^m \P\left(\left|(\varphi^{\eta},-\Delta s_j)\right|\ge c_02^{-j/8}r \right) \\
    &\le C\sum _{j=1}^m \exp \big(-cr^{4-2H} 2^{j/4} \big) +C \sum _{j=1}^m 2^{dj}q(2^{j/8}r) \\
    &\le C\exp \big(-cr^{4-2H} \big) +C \sum _{j=1}^m 2^{dj}q(2^{j/8}r).
\end{split}
\end{equation}
This finishes the proof of the lemma using a union bound over the coordinates.
\end{proof}

\begin{proof}[Proof of the upper bound in \eqref{eq:loweboundheight d>4} from Theorem \ref{thm:main d>4}]
    Fix $r\ge 1$. One can check, using induction on $\ell >0$ that 
    \begin{equation}\label{eq:3}
        q(r)\le \sum _{i=0}^\ell C^{i+1} 2^{2id} \exp (-c(r2^{i/8})^{4-2H}) +\sum _{j=1}^{ \lceil \log _2L \rceil } C^{\ell +1 } 2^{d(j+2\ell )}q(r2^{(j+\ell )/8}),
    \end{equation}
    where $C,c$ are the constants from Lemma~\ref{lem:loc}. Indeed, the case $\ell =0$ follows from Lemma~\ref{lem:loc} and the induction step follows by substituting the bound in Lemma~\ref{lem:loc} instead of  $q(r2^{(\ell+1)/2})$ appearing in the term corresponding to $j=1$ on the right hand side of \eqref{eq:3}. 

    Next, note that the first sum on the right hand side of \eqref{eq:3} converges and that $q$ is monotonically decreasing. Thus, by slightly changing $C,c$ we obtain for all $\ell\ge 0$ that
    \begin{equation}\label{eq:4}
        q(r)\le C \exp (-cr^{4-2H})+CL^C 2^{2d\ell}q(r2^{\ell /8}). 
    \end{equation}
    It follows from Corollary~\ref{cor:maxphi} that the second term in the bound \eqref{eq:4} tends to $0$ as $\ell \to \infty$ and therefore we obtain the bound  $q(r)\le C \exp (-cr^{4-2H})$.
\end{proof}

\subsubsection{Dimension $d=4$}

In this section we prove the upper bound in \eqref{eq:bound h d=4} from Theorem \ref{thm: main 4}. When $d=4$ we prove the following recursive estimate on the function $q$.

\begin{lem}
    Suppose that $d=4$. Then, for all $r>1$ we have
        \begin{equation}
        q(r) \le C(\log L)\exp(-cr^{4-2H} /\log ^4L)   +C L^{d} (\log L ) q(2r)  
    \end{equation}
\end{lem}

\begin{proof}
    We use the shift functions $s_j$ given in \eqref{def:s_j}. By Lemma~\ref{lem:fluctuation and concentration} with $h$ being $2r$ we have 
    \begin{equation}
 \P \big( \left|(\varphi^{\eta},-\Delta s_j)\right|\ge r/(md) \big) \le C\exp \big(-cr^{4-2H}/m^4 \big)+ CL^{d} q(2r).
\end{equation}
Thus, we obtain 
\begin{equation}
    \mathbb P \big( | \varphi ^\eta  _v \cdot e | \ge r/d \big) \le Cm \exp \big( -cr ^{4-2H}/m^4 \big) +CmL^{d} q(2r).
\end{equation}
This finishes the proof of the lemma using a union bound over the coordinates.
\end{proof}

\begin{proof}[Proof of the upper bound in \eqref{eq:bound h d=4} from Theorem \ref{thm: main 4}]
Using induction on $\ell \ge 0$ we obtain that 
     \begin{equation}
        q(r) \le  L^{2d\ell }q(r2^\ell )+ \sum _{i=0}^{\ell-1} C L^{2d(i+1)} \exp (-c(r2^i)^{4-2H}/\log ^4L).
    \end{equation}
    Thus, taking $\ell \to \infty $ and using Corollary~\ref{cor:maxphi} we obtain that for all $t$ sufficiently large $q(t(\log L)^{5/(4-2H)}) \le CL^{-ct^{4-2H}}$, finishing the proof. 
\end{proof}

\subsubsection{Dimensions $d\le 3$}
In this section we prove the upper bound in \eqref{eq:123 height fluctuations thm} from Theorem \ref{thm:main 123}. For any $h>0$ define 
\begin{equation}
    p(h):=\mathbb P \big( \max _{v\in \Lambda } \| \varphi _v\| \ge h  \big).
\end{equation}

\begin{lem}\label{lem:loc}
    Suppose that $d\le 3$. For any $v\in \Lambda _L$ and $h\ge r> 0$ we have
    \begin{equation}\label{eq:loclemma}
       \mathbb P \big( \| \varphi ^\eta  _v \| \ge r  \big) \le  C\exp \big( -cr^4h^{-2H} r_v^{-1/2}L^{d-7/2} \big) +3\sum _{j=0}^{\lceil \log_2 L \rceil } p(2^{j/20}h).
    \end{equation}    
\end{lem}

\begin{proof}[Proof of Lemma~\ref{lem:loc}]
    Fix $v\in \Lambda _L$ and a unit vector $e\in \mathbb R ^n$. Without loss of generality, we may assume that $r^4h^{-2H} r_v^{-1/2}L^{d-7/2} \ge 1$ as other wise the statement of the lemma clearly holds. We consider the shift functions $s_j=s_j^v$ defined in \eqref{def:s_j}. Recall that $G_m^ve=\sum _{j=1}^ms_j$ and that
    \begin{equation}\label{eq:phi i3}
        \varphi ^{\eta } _v \cdot e=(\varphi ^\eta ,-\Delta G_m^ve)= \sum _{j=1}^m (\varphi^ \eta ,- \Delta s_j).
    \end{equation}
    Next, we fix a constant $c_0>0$ such that $c_0\sum _{j=0}^{\infty } 2^{-j/20}<1/d$. By Claim~\ref{claim:D(s)}, we have that for all $1\le j \le m$
    \begin{equation*}
        D(s_j) \le Cr_v^{4-2H}2^{j(2H-d)}\le Cr_v^{1/2}2^{j(2H-d+7/2-2H)}=C r_v^{1/2}2^{(j-m+m)(7/2-d)}\le C r_v^{1/2}L^{7/2-d} 2^{(j-m)/2}
    \end{equation*}
   where in the second inequality we used that $r_v\le 2^j $ and that $7/2-H\ge 0$ (since $H\le 1$); in the last inequality we used that $2^m \le 2L$ and $7/2-d\ge 1/2$ for $d\le 3$. Thus, by Lemma~\ref{lem:fluctuation and concentration} with $h$ being $2^{(m-j)/20}h$ and $r$ being $c_02^{(j-m)/20}r$ we have  
    \begin{equation}\label{1}
        \mathbb P \big( | (\varphi^ \eta , \Delta s_j)| \ge c_02^{(j-m)/20}r \big) \le  C\exp \big( -cr^4h^{-2H} r_v^{-1/2} L^{d-7/2} 2^{(m-j)/2}  \big) +3p(2^{(m-j)/20}h)   .
    \end{equation}

    Thus, by \eqref{eq:phi i3} and the choice of $c_0$ we obtain
\begin{equation}\label{eq:77733}
\begin{split}
    \mathbb P \big( |\varphi ^\eta  _v \cdot e| \ge r/d  \big) &\le \sum _{j=1}^m \mathbb P \big( | (\varphi^ \eta , \Delta s_j)| \ge c_02^{(j-m)/20}r  \big) \\
    &\le C \sum _{j=1}^m\exp \big( -cr^4h^{-2H} r_v^{-1/2}L^{d-7/2} 2^{(m-j)/2} \big) + 3\sum _{j=1}^m p(2^{(m-j)/20} h) \\
    &\le C\exp \big( -cr^4h^{-2H} r_v^{-1/2}L^{d-7/2} \big) +3\sum _{j=0}^{m-1} p(2^{j/20}h).
\end{split}
\end{equation}
This finishes the proof of the lemma.
\end{proof}

Looking at the last proof, we observe that the box $\Lambda _L$ could be replaced with any other box $\Delta \subseteq \Lambda _L$ with $v\in \Delta $. However, $\varphi ^\eta |_{\Delta ^c}$ is not necessarily $0$ when $\Delta \neq \Lambda _L$ and therefore the first equality in \eqref{eq:phi i3} does not hold and the expression $\varphi ^\eta _v\cdot e$ should be replaced with $(\varphi ^\eta , -\Delta G_\Delta ^v e)$. Thus, if $\Delta \subseteq \Lambda _L$ is a box of side length $\ell $ and $d(v,\Delta ^c)\ge \ell /4$, we obtain that for all $h\ge r >0$
\begin{equation}\label{eq:with delta}
       \mathbb P \big( |(\varphi ^\eta , \Delta G_\Delta ^v \cdot e) | \ge r  \big) \le  C\exp \big( -cr^4h^{-2H} \ell ^{d-4} \big) +3\sum _{j=0}^{ \lceil \log_2 \ell \rceil } p(2^{j/20}h).
\end{equation}

In the following lemma we bound the height difference between pairs of vertices in $\Lambda $. We will then combine this estimate with a chaining argument in order to bound the maximal height of the surface.

\begin{lem}\label{lem:holder}
    Suppose that $d\le 3$. Then, for all $v,u\in \mathbb Z ^d$ and $h\ge r>0$ we have  
    \begin{equation}\label{eq:holder}
        \mathbb P \big( \| \varphi^ \eta _v-\varphi^ \eta _u\| > r \big) \le C\exp \big( -cr^4h^{-2H} \|u-v\|^{-1/2}L^{d-7/2} \big) +9\sum _{j=0}^m p(2^{j/20}h).
    \end{equation}
\end{lem}

\begin{proof}[Proof of Lemma~\ref{lem:holder}]
    The proof is similar to the proof of Lemma~\ref{lem:loc} and some of the details are omitted. Fix $u,v\in \mathbb Z ^d$ and a unit vector $e\in \mathbb R ^n$. We may assume that $u,v\in \Lambda_L $ and that $\|u-v\|\le r_v/10$. Indeed, if $\|u-v\|\ge  r_v/10$ then the inequality in \eqref{eq:holder} follows from Lemma~\ref{lem:loc}. 

    For the proof of this lemma we work with slightly different shift functions. Let $\Lambda _j:=\Lambda_L  \cap (v+(-2^j,2^j)^d)$ and let $m$ be the first integer for which $\Lambda _{m}=\Lambda _L$. Let $m':=\lceil  \log _2 \|u-v\| \rceil +2$ and note that both $u,v$ are in $\Lambda _{m'}$. For $m'\le j \le m$ we consider the Green's functions $G_j^v:=G_{\Lambda _j}^v$ and $G_j^u:=G_{\Lambda _j}^u$, both with respect to the same box centered at $v$. For $m'< j \le m$ define the functions $s_j'$ by 
    \begin{equation}
        s'_j(x):= (G_j^v(x)-G_j^u(x))e-(G_{j-1}^v(x)-G_{j-1}^u(x))e
    \end{equation}
and note that 
    \begin{equation}\label{eq:phi2}
        (\varphi^ \eta _v-\varphi^ \eta _u)\cdot e =-(\varphi ^\eta ,\Delta G_j^v(x)\cdot e )+ (\varphi ^\eta ,\Delta G_j^u(x)\cdot e ) - \sum _{j=m'+1}^{m} (\varphi^ \eta ,\Delta s_j').
    \end{equation}

Next, as in Claim~\ref{claim:D(s)}, we prove that for all $m<j\le m$ we have that
\begin{equation}\label{eq:D(s_j')}
    D(s'_j) \le C\|u-v\|^{4-2H} 2^{j(2H-d)} \le C\|u-v\|^{1/2}2^{j(7/2-d)} \le C\|u-v\|^{1/2}L^{7/2-d}2^{(j-m)/2},
\end{equation}
where in the second inequality we used that $H\le 1$ and the last inequality we used $d\le 3$. To this end, it suffices to prove that 
        \begin{equation}\label{eq:sj2}
        \|s_j'\|_{2H}^{2H} \le C\|u-v\|^{2H}2^{dj+2H(1-d)j}   , \quad  \|\nabla s_j'\|_\Lambda ^2 \le C\|u-v\|^22^{-dj} 
    \end{equation} 
By the second part of Corollary~\ref{cor:green} (with $\Lambda=\Lambda_j$ and $\Delta=\Lambda_{j-1}$) we have $\max _{x\in \Lambda_j }\|s_j'(x)\| \le \|u-v\|2^{j(1-d)}$ and therefore $\|s_j'\|_{2H}^{2H}\le \|u-v\| ^{2H}2^{2H(1-d)j +dj}$. This finishes the proof of the first part of \eqref{eq:sj2}.

Next, for all $m'<j \le m$ we have 
    \begin{equation}
        \|\nabla s_j'\| _\Lambda ^2 = - \sum _{x\in \partial \Lambda _{j-1}\setminus\partial \Lambda } \! (G_j^v(x)-G_j^u(x)) \Delta (G_{j-1}^v(x)-G_{j-1}^u(x))\le C\|u-v\|^22^{-jd}
    \end{equation}
    where in here we used that, by Lemma~\ref{lem:Green}, for all $x\in \partial \Lambda _{j-1}\setminus\partial \Lambda_L$ we have $|G_j^v(x)-G_j^u(x)|\le C\|u-v\|2^{j(1-d)}$ and $|\Delta (G_{j-1}^v(x)-G_j^u(x))|=| G_{j-1}^v(y)-G_j^u(y)|\le C\|u-v\|2^{-jd}$ where $y$ is the unique neighbour of $x$ in $\Lambda _{j-1}$. This finishes the proof of \eqref{eq:sj2}.

Next, fix a constant $c_0>0$ such that $c_0\sum _{j=0}^\infty 2^{-j/20} \le 1/(3d)$. By \eqref{eq:D(s_j')} and Lemma~\ref{lem:fluctuation and concentration} with $h$ being $2^{(m-j)/20}h$ we have
    \begin{equation*}
        \mathbb P \big( | (\varphi^ \eta , \Delta s_j)| \ge c_02^{(j-m)/20}r \big) \le  C\exp \big( -cr^4h^{-2H} \|u-v\|^{-1/2} L^{d-7/2} 2^{(m-j)/2}  \big) +3 p(2^{(m-j)/20}h).
    \end{equation*}
Thus, by \eqref{eq:phi2} and the choice of $c_0$ we have
\begin{equation*}
\begin{split}
    \mathbb P \big( |(\varphi ^\eta  _v -\varphi _u^\eta ) \cdot e| \ge r/d  \big) & \le \mathbb P \big( \big| (\varphi^ \eta , (\Delta  G_{m'}^v)e ) \big| \ge r/(3d) \big)+ \mathbb P \big( \big| (\varphi^ \eta , (\Delta  G_{m'}^u)e ) \big| \ge r/(3d) \big) \\
    &\quad \quad \quad \quad \quad \quad \quad + \sum _{j=m'+1}^m \mathbb P \big( | (\varphi^ \eta , \Delta s_j)| \ge c_02^{(j-m)/20}r  \big) \\
    &\le C\exp \big( -cr^4h^{-2H} \|u-v\|^{-1/2}L^{d-7/2} \big) +9\sum _{j=0}^m p(2^{j/20}h),
\end{split}
\end{equation*}
where in the last inequality we used \eqref{eq:with delta} to bound the first two terms and completed the calculation just like in \eqref{eq:77733}.
\end{proof}

\begin{cor}\label{cor:p}
    Suppose that $d\le 3$. For any $r>0$ we have that 
    \begin{equation}
        p(r)\le C\exp (-cr^{4-2H}L^{4-d}) + C\sum _{j=1}^{4\log L} 2^{dj}p(2^{j/20}r).
    \end{equation}
\end{cor}

\begin{proof}
Let $r>0$. We may assume that $r^{4-2H}L^{d-4}\ge 1$ as otherwise the statement clearly holds (with appropriate constants $C,c$). Let $m:=\lfloor \log _2 L \rfloor +1$ and fix $c_0>0$ such that $c_0\sum _{j=0}^\infty 2^{-j/20}<1$. Define the event
    \begin{equation}
        \mathcal E := \bigcap _{j=0}^{m} \Big\{ \forall u,v\in 2^j \mathbb Z ^d , \|u-v\|_\infty \le 2^{j}, \  \| \varphi^ \eta _u -\varphi^ \eta _v \| \le c_02^{(j-m)/20}r \Big\}.
    \end{equation}
    Next, we fix $v\in \Lambda _L$ and show that $\| \varphi^ \eta _v \| \le r$ on the event $\mathcal E$, and therefore $p(r)\le \mathbb P (\mathcal E ^c)$. Without loss of generality, suppose that all the coordinates of $v$ are non-negative. There exists a sequence of vertices $v=v_0,v_1,\dots ,v_m$ that is non decreasing in each coordinate such that the following holds. For all $j<m$ we have that $v_j\in 2^j\mathbb Z ^d$ and $v_{j+1}-v_j\in \{ 0,2^j \}^d$ and $v_m\notin \Lambda _L$. Thus, using that $\varphi^ \eta _{v_m}=0$ we have on the event $\mathcal E $
    \begin{equation}
        \| \varphi^ \eta _v \| \le \sum _{j=0}^m \| \varphi^ \eta _{v_{j}}-\varphi^ \eta _{v_{j-1}} \| \le c_0r\sum _{j=0}^m 2^{(j-m)/20} \le r.
    \end{equation}
We turn to bound the probability of $\mathcal E ^c$. To this end, let $A_j$ be the set of pairs of vertices $u,v\in 2^{m-j} \mathbb Z ^d$, with at least one of them in $\Lambda _L$ such that $\|u-v\| _\infty \le 2^{m-j}$. We have that $|A_j|\le C2^{dj}$. Moreover, by Lemma~\ref{lem:holder} with $h$ being $2^{(j+1)/20}r$, for any $(u,v)\in A_j$ we have
\begin{equation*}
    \mathbb P \big( \| \varphi^ \eta _u -\varphi^ \eta _v \| \ge  c_02^{-j/20}r \big)\le C\exp \big( -c r^{4-2H} L^{d-4}2^{j/5}  \big) +9\sum _{i=0}^{m} p(2^{(i+j+1
    )/20}r),
\end{equation*}
where in here we also used $H\le 1$. Thus, we obtain
\begin{equation}
\begin{split}
    \mathbb P ( \mathcal E ^c ) &\le \sum _{j=0}^m \sum _{(u,v)\in A_j} \mathbb P \big( \| \varphi^ \eta _u -\varphi^ \eta _v \| \ge  c_02^{-j/20}r \big) \\
    &\le C \sum _{j=0}^{m} 2^{dj}  \exp \big( -cr^{4-2H} L^{d-4} 2^{j/5} \big) +C\sum _{j=0}^m 2^{dj} \sum _{i=0}^{m}  p(r2^{(i+j+1)/20})\\
    &\le C \exp \big( -cr^{4-2H} L^{d-4} \big) +C\sum _{j=1}^{2m} 2^{dj} p(2^{j/20}r),
\end{split}
\end{equation}
finishing the proof of the corollary.
\end{proof}

Note that $p$ and $q$ satisfy similar recursion inequalities.
Hence, the upper bound in \eqref{eq:123 height fluctuations thm} from Theorem \ref{thm:main 123} follows from Corollary~\ref{cor:p} in the same way as in the proof of the upper bound in \eqref{eq:loweboundheight d>4} from Theorem \ref{thm:main d>4}.

\subsection{Ground energy concentration}

In this section, we prove the upper bounds on the ground energy fluctuations given in \eqref{eq:lowerbound26}, \eqref{eq:boundGE4} and \eqref{eq:boundGEd>4} (Theorems \ref{thm:main 123}, \ref{thm: main 4} and \ref{thm:main d>4}) using the previously obtained upper bounds on the heights of the minimal surface. The upper bounds on the standard deviations in \eqref{eq:bound height 123} \eqref{eq:27}, \eqref{eq:GE d>4} will then follow easily.

To this end we will need the following claim.

\begin{claim}\label{claim:average}
Let $X_1,\dots ,X_n$ be a sequence of non-negative random variables (not necessarily independent) such that for all $t>0$
\begin{equation}\label{as:1}
    \max _{i\le n} \mathbb P \big( X_i \ge t\big) \le c_1e^{-c_2t^\alpha },
\end{equation}
for some constants $c_1,c_2,\alpha >0$. Then, there are constants $C=C(c_1,c_2,\alpha )$ and $c=c(c_1,c_2,\alpha )$ such that for all $t>0$
\begin{equation}
    \mathbb P \Big(\frac{1}{n} \sum _{i=1}^n X_i \ge t\Big) \le Ce^{-ct^\alpha }.
\end{equation}
\end{claim}

\begin{proof}
    Let $c_3:=c_2/2$ and note by the assumption in \eqref{as:1} we have
    \begin{equation*}
        \mathbb E [ e^{c_3X_i^\alpha }] = \int _0^{\infty } \!\mathbb P (e^{c_3X_i^\alpha }\ge t)dt= \int _0^{\infty } \! \mathbb P \big( X_i \ge ((\log t)/c_3) ^{1/\alpha } \big) dt \le \int _0^{\infty } \!\! \min \big(1,\frac{c_1}{t^2}\big)dt = 2\sqrt{c_1}. 
    \end{equation*}
    Next, let $c_4>0$ sufficiently large so that the function $g$ defined by $g(x):=\max (c_4,e^{c_3x^{\alpha }})$ is convex on $(0,\infty )$. It follows that $\mathbb E [g(X_i)] \le c_4+2\sqrt{c_1}$ and therefore by Jensen's inequality we have
    \begin{equation}
        \mathbb E \Big[ g\Big( \frac{1}{n}\sum _{i=1}^nX_i \Big) \Big] \le \mathbb E \Big[ \frac{1}{n}\sum _{i=1}^ng(X_i) \Big] \le c_4+2\sqrt{c_1}.
    \end{equation}
    Thus, by Markov's inequality, for all $t$ sufficiently large we have
    \begin{equation*}
        \mathbb P \Big(\frac{1}{n} \sum _{i=1}^n X_i \ge t\Big) \le \mathbb P \Big( g\Big( \frac{1}{n} \sum _{i=1}^n X_i \Big) \ge g(t)\Big) \le e^{-c_3t^\alpha } \mathbb E \Big[ g\Big( \frac{1}{n}\sum _{i=1}^nX_i \Big) \Big] \le (c_4+2\sqrt{c_1})e^{-c_3t^\alpha }.
    \end{equation*}
    This finishes the proof of the claim.
\end{proof}

 As usual, we write $\GE ^{\eta }$ for $\GE ^{\eta ,\Lambda _L}$.

\begin{proof}[Proof of the upper bound in \eqref{eq:lowerbound26} from Theorem \ref{thm:main 123}]
 Let $r>0$ and let $h:=\sqrt{r}L^{\frac{2-d}{2}}$. By Proposition~\ref{as:conc} with $\Delta =\Lambda _L$ and the upper bound \eqref{eq:123 height fluctuations thm} with $t:=h L^{\frac{d-4}{4-2H}}$ we have
\begin{equation*}
\begin{split}
    \inf _{\gamma \in \mathbb R}\mathbb P \big(  |\GE ^{\eta } -\gamma & | \ge r \big) \le 2 \exp \Big(- \frac{cr^2}{h^{2H}|\Lambda _L|} \Big) +\mathbb P \Big(  \max _{v\in \Lambda _L } \|\varphi ^\eta  _v\| \ge h \Big)\\
    &\le  C\exp \Big( -\frac{cr^2}{h^{2H}L^d} \Big) +C\exp \Big( -\frac{ch^{4-2H}}{L^{4-d}} \Big) \le C\exp \Big( -\frac{cr^{2-H}}{ L^{d-dH+2H} } \Big)  .
\end{split}
\end{equation*}
Since this inequality holds for all $r>0$ it follows that $\mathbb E |\GE^{\eta }|<\infty $ and that for all $r>0$
\begin{equation}
    \mathbb P \big(  |\GE ^{\eta} -\mathbb E [\GE^{\eta} ]| \ge r \big) \le C\exp \Big( -\frac{cr^{2-H}}{ L^{d-dH+2H} } \Big),
\end{equation}
for slightly different constants $C,c>0$ (see, e.g., \cite[Claim~5.6]{dembin2024minimal}). This finishes the proof.
\end{proof}

When $d\ge 4$ the maximal height of $\varphi ^\eta $ is asymptotically larger than the height $\|\varphi _v\|$ at a typical vertex $v\in \Lambda _L$. Hence, in order to obtain better bounds in this case we use the first and stronger bound given in Proposition~\ref{as:conc} together with Claim~\ref{claim:average}.
\begin{proof}[Proof of the upper bound in \eqref{eq:boundGE4} from Theorem \ref{thm: main 4}]
 For $v\in \Lambda _L$, define the random variable $X_v:=\|\varphi ^\eta _v\|^{2H} (\log L)^{\frac{-10H}{4-2H}}$ and observe that thanks to the upper bound in \eqref{eq:bound h d=4} we have $\mathbb P (X_v\ge t) \le Ce^{-ct^{\alpha }}$ for any $t>0$ with $\alpha :=(2-H)/H$. Thus, by Claim~\ref{claim:average} for any $h>0$ we have
\begin{equation*}
    \mathbb P \Big(  \frac{1}{|\Lambda _L|}\sum _{v\in \Lambda _L } \|\varphi ^\eta  _v\|^{2H} \ge h^{2H} \Big) = \mathbb P \Big(  \frac{1}{|\Lambda _L|}\sum _{v\in \Lambda _L } X_v \ge h^{2H} (\log L)^{\frac{-10H}{4-2H}} \Big) \le C\exp \Big( -\frac{ch^{4-2H}}{\log ^5L} \Big).
\end{equation*}
Hence, letting $h:=\sqrt{r}L^{-1}(\log L)^{5/4}$ we have by Proposition~\ref{as:conc} that
\begin{equation*}
\begin{split}
    \inf _{\gamma \in \mathbb R}\mathbb P \big(  |\GE ^{\eta } -\gamma & | \ge r \big) \le 2 \exp \Big(- \frac{cr^2}{h^{2H}|\Lambda _L|} \Big) +\mathbb P \Big(  \frac{1}{|\Lambda _L|}\sum _{v\in \Lambda _L } \|\varphi ^\eta  _v\|^{2H} \ge h^{2H} \Big)\\
    &\le  C\exp \Big( -\frac{cr^2}{h^{2H}L^4} \Big) +C\exp \Big( -\frac{ch^{4-2H}}{\log^5 L} \Big) \le C\exp \Big( -\frac{cr^{2-H}}{ L^{4-2H} (\log L)^{\frac{5H}{2}} } \Big).
\end{split}
\end{equation*}
This finishes the proof using the same arguments as the proof above.
\end{proof}

\begin{proof}[Proof of the upper bound in \eqref{eq:boundGEd>4} from Theorem \ref{thm:main d>4}]
    For any $r>0$ we let $h:=\sqrt{r}L^{-d/4}$ and similarly as in the proof of \eqref{eq:boundGE4} use Proposition~\ref{as:conc}, \eqref{eq:loweboundheight d>4} and Claim~\ref{claim:average} to obtain
\begin{equation*}
\begin{split}
    \inf _{\gamma \in \mathbb R}\mathbb P \big(  |\GE ^{\eta } -\gamma & | \ge r \big) \le 2 \exp \Big(- \frac{cr^2}{h^{2H}|\Lambda _L|} \Big) +\mathbb P \Big(  \frac{1}{|\Lambda _L|}\sum _{v\in \Lambda _L } \|\varphi ^\eta  _v\|^{2H} \ge h^{2H} \Big)\\
    &\le  C\exp \Big( -\frac{cr^2}{h^{2H}L^d} \Big) +C\exp \big( -ch^{4-2H} \big) \le C\exp \Big( -c r^{2-H} L^{-d(2-H)/2} \Big).
\end{split}
\end{equation*}
This finishes the proof.
\end{proof}

\section{Delocalization}\label{sec:deloc}
In this section, we derive the lower bounds for the height and ground energy fluctuations stated in Theorems \ref{thm:main 123}, \ref{thm: main 4} and \ref{thm:main d>4}. The core of the proof relies on the following theorem concerning local delocalization, which asserts that significant deviations of a simple local observable with respect to the noise within a subset \(\Lambda \subset \Lambda_L\) result in the delocalization of a positive density of vertices on the surface above \(\Lambda\). The lower bounds for the height fluctuations will be a straightforward application of this theorem in dimensions $d\ne 4$, and will be an application of this theorem together with a fractal-percolation argument in dimension $d=4$. The theorem will also be the main ingredient in the proof of the lower bounds for the ground energy fluctuations in dimensions $d\in\{1,2,3\}$.

\begin{thm}
\label{theorem:83-1} Let $d\ge 1$ and $H\in (0,1)$. There are constants $C,c,\alpha >0$ depending only on $d,H$ such that the following
holds. Let $L\ge 3$, $h\ge 1$ and 
$\Lambda=v_0+\Lambda_{\ell}\subseteq\Lambda_{L}$. Set 
\begin{equation}
    \hat \eta_\Lambda :=\sum _{v\in \Lambda} \eta_{v,\alpha h e_1}+\eta_{v,-\alpha h e_1}.
\end{equation}
Then
\begin{equation}
\mathbb{E}\big[\big|\big\{ v \in \Lambda  : \| \varphi _v^\eta  \| >  h \big\}\big|\mid\hat{\eta}_{\Lambda},\mathcal{F}_{\Lambda^{c}}\big]\ge c|\Lambda|\quad\text{on the event}\quad\big\{\hat{\eta}_{\Lambda}\le - Ch^2\ell ^{d-2}\big\}
\end{equation}
where $\mathcal F _{\Lambda ^c}$ denotes the sigma algebra generated by the disorder in $\Lambda ^c \times \mathbb R ^n$.
\end{thm}

Sections~\ref{sec:det}, \ref{sec:rand} and \ref{sec:proof} below are devoted to the proof of Theorem~\ref{theorem:83-1}. Section~\ref{sec:det} provides deterministic inputs required for the proof and Section~\ref{sec:rand} provides probabilistic inputs. In Section~\ref{sec:proof}, we use those ingredients in order to complete the proof.

In Sections~\ref{sec:dneq4} and \ref{sec:G} we use Theorem~\ref{theorem:83-1} in order to prove the main delocalization results of the paper (that is, the lower bounds in Theorems \ref{thm:main 123}, \ref{thm: main 4} and \ref{thm:main d>4}).

\subsection{Deterministic results}\label{sec:det}
Let $\Pi\subset\Omega^{\Lambda}$ be a closed set of surfaces. For an arbitrary environment $\eta$ such that $H^{\eta}$ admits a (finite) minimum on $\Pi$ we denote $\varphi^{\eta,\Pi}\coloneqq\arg\min_{\varphi\in\Pi}H^{\eta}(\varphi)$
and $\GE_{\Pi}^{\eta}\coloneqq H^{\eta}(\varphi^{\eta,\Pi})$. In here if there is more than one minimizer, we let $\varphi^{\eta,\Pi}$ be the minimizer that is the first in, say, lexicographic order (as a vector in $(\mathbb R ^n)^{\mathbb Z ^d}$). The following proposition is the main deterministic input of our proofs of delocalization. It gives a sufficient condition to have a delocalized surface. Roughly speaking, it says that if the gain in terms of noise of being shifted by $s$ compensates the extra cost of the shift in terms of gradient, then the surface has to delocalize.
\begin{prop}\label{prop:Deloc2}
Let $\Lambda\subset\Z^{d}$. Let $\Pi\subset\Omega^{\Lambda}$ be a closed set.
Let $s:\Z^{d}\to\R^{n}$ be supported on $\Lambda$. Let $\eta$ and $\zeta$
be two environments such that $\eta ,\zeta$ and the shifted versions $\zeta^s,\zeta^{-s}$ admit a minimum on $\Omega ^{\Lambda,\tau}$ and $\Pi$. Denote 
\begin{align*}
\Delta_{1} & \coloneqq\GE_{\Pi}^{\eta}-\GE_{\Pi}^{\zeta}\\
\Delta_{2} & \coloneqq\inf_{\varphi\in (\Pi+s)\cup(\Pi-s)}H^\zeta(\varphi)-H^\eta(\varphi)\\
\Delta_{3} & \coloneqq\frac{1}{2}(\varphi^{\zeta^{s}}-\varphi^{\zeta^{-s}},-\Delta s).
\end{align*}
Then
\begin{equation}\label{eq:deterministic delocalization}
\{\Delta_1+\Delta_2+\Delta_3>\|\nabla s\|_{\Lambda}^{2}\}\subset \{\varphi^{\zeta^{s}}\notin\Pi\}\cup \{\varphi^{\zeta^{-s}}\notin\Pi\}\cup  \{\varphi^{\eta}\notin\Pi\}.
\end{equation}
\end{prop}

One advantage of the quantity $\Delta_3$ is that it may be controlled via Markov's inequality for general sets $\Pi$; see Corollary~\ref{cor:markov}.  

In several places, we will make use of the shift function of the following proposition. We refer to \cite[Proposition~4.2]{dembin2024minimal} for the proof.
\begin{prop}
\label{prop:const_laplace}Fix $d\in\N$, $\ep>0$. Let $L\ge 1$.
Write $\Lambda_{L}^{-}=\Lambda_{\left\lceil \left(1-\frac{\ep}{2d}\right)L\right\rceil }$
and note that $|\Lambda_{L}^{-}|/|\Lambda_{L}|\ge1-\ep$. There is
a function $\pi:\Z^{d}\to\R$ satisfying the following conditions:
\begin{align}
\pi_{v} & =0 & \,\,\,\,\forall v\notin\Lambda_{L}\label{eq:s_0}\\
\Delta\pi_{v} & \le C_{\ep,d}L^{-2} & \,\,\,\,\forall v\in\Z^{d}\label{eq:s_laplace}\\
\pi_{v} & \ge 1, & \,\,\,\,\forall v\in\Lambda_{L}^{-}\label{eq:s_min}\\
\|\nabla\pi\|^{2} & \le C_{\ep,d}L^{d-2}.\label{eq:s_energy}
\end{align}
\end{prop}

\begin{proof}[Proof of Proposition \ref{prop:Deloc2}]
    We have
\begin{equation}\label{eq:application main identity 1}
\begin{split}
    \GE^{\eta}\le H^{\eta}(\varphi^{\zeta^s,\Pi}-s)&\le  H^{\zeta}(\varphi^{\zeta^s,\Pi}-s)-\Delta_2\\
    &= \GE^{\zeta^s}_\Pi+ (\varphi^{\zeta^s,\Pi},\Delta s)+ \frac{1}{2}\|\nabla s\|_\Lambda^2-\Delta_2+ \sum_{v\in\Lambda}\zeta_{v,-s_v}
    \end{split}
\end{equation}
where we used Proposition~\ref{prop:main identity} in the last equality replacing $\varphi$ by $\varphi^{\zeta^s,\Pi}-s$ and $\eta$ by $\zeta$.
Similarly, by replacing $s$ by $-s$ in the previous inequality
\begin{equation*}
\begin{split}
    \GE^{\eta}\le  H^{\zeta}(\varphi^{\zeta^{-s},\Pi}+s)-\Delta_2 =\GE^{\zeta^{-s}}_\Pi+ (\varphi^{\zeta^{-s},\Pi},-\Delta s)+ \frac{1}{2}\|\nabla s\|_\Lambda^2-\Delta_2+ \sum_{v\in\Lambda}\zeta_{v,s_v}.
    \end{split}
\end{equation*}
Summing the two previous inequalities, we get
\begin{equation}\label{eq:000}
    2\GE^{\eta}\le \GE^{\zeta^s}_\Pi+\GE^{\zeta^{-s}}_\Pi+ \|\nabla s\|_\Lambda^2-2\Delta_2-(\varphi^{\zeta^{s},\Pi}-\varphi^{\zeta^{-s},\Pi},-\Delta s)+ \sum_{v\in\Lambda}(\zeta_{v,s_v}+\zeta_{v,-s_v}).
\end{equation}
Moreover, we have 
\begin{equation}
\GE^ {\zeta^{s}}\le H^ {\zeta^ s}(\varphi^ {\zeta ,\Pi}+s)= \GE^ {\zeta}_\Pi+(\varphi^ {\zeta ,\Pi},-\Delta s)+\frac{1}{2}\|\nabla s\|_\Lambda^2- \sum_{v\in\Lambda}\zeta_{v,-s_v}
\end{equation}
where we used Proposition~\ref{prop:main identity} in the last equality.
Similarly, by replacing $s$ by $-s$ in the previous inequality
\begin{equation}
\GE^ {\zeta^{-s}}\le \GE^ {\zeta}_\Pi+(\varphi^ {\zeta ,\Pi},\Delta s)+\frac{1}{2}\|\nabla s\|_\Lambda^2-  \sum_{v\in\Lambda}\zeta_{v,s_v}.
\end{equation}
So summing the two previous inequalities yields
\begin{equation}\label{eq:001}
\begin{split}
\GE^ {\zeta ^{ s}}+ \GE^ {\zeta ^{- s}}&\le 2\GE^ {\zeta}_\Pi+\|\nabla s\|_\Lambda^2- \sum_{v\in\Lambda}(\zeta_{v,s_v}+\zeta_{v,-s_v})\\
&\le 2\GE^ {\eta}_\Pi- 2\Delta_1 +\|\nabla s\|_\Lambda^2- \sum_{v\in\Lambda}(\zeta_{v,s_v}+\zeta_{v,-s_v}).
\end{split}
\end{equation}
Finally, on the event $\{\varphi^{\zeta^{s}}\in\Pi\}\cap \{\varphi^{\zeta^{-s}}\in\Pi\}$, we get combining inequalities \eqref{eq:000} and \eqref{eq:001}
\begin{equation}
    2\GE^\eta\le 2\GE^\eta_\Pi +2\|\nabla s\|_\Lambda^2-2(\Delta_1+\Delta_2+\Delta_3).
\end{equation}
Note that $ \GE^\eta< \GE^\eta_\Pi$ implies that $\varphi^{\eta}\notin\Pi$. 
It follows that 
\[\{\Delta_1+\Delta_2+\Delta_3>\|\nabla s\|_{\Lambda}^{2}\}\subset \{\varphi^{\zeta^{s}}\notin\Pi\}\cup \{\varphi^{\zeta^{-s}}\notin\Pi\}\cup  \{\varphi^{\eta}\notin\Pi\}.\]
The result follows.
\end{proof}

\subsection{Probabilistic results}\label{sec:rand}

We first introduce some useful results. For $s\in\Omega^ \Lambda$,
define $H_{s}^{\eta}(\varphi)\coloneqq H^{\eta^{-s}}(\varphi-s)$.
Accordingly, let $\varphi_{s}^{\eta}\coloneqq\varphi^{\eta^{-s}}+s$ be the
minimizer of $H_{s}^{\eta}(\varphi)$. 
\begin{prop}
\label{prop:close_surf_markov}Fix $\Lambda$ and $\eta$ be a random environment.
Let $s\in\Omega^{\Lambda}$. Denote $D=(\varphi_s^{\eta}-\varphi^{\eta},-\Delta s\,)$. Assume $\eta^s$ and $\eta$ have the same distribution.
Then, $0\le H^{\eta}(\varphi_{s}^{\eta})-H^{\eta}(\varphi^{\eta})\le D$
and for every $a>0$ 

\begin{align*}
\P(\,D & \le a\|\nabla s\|^{2}\,)\ge1-\frac{1}{a}.
\end{align*}
\end{prop}
The following corollary follows easily by applying the previous proposition to the disorder~$\eta^s$ and the shift $2s$ and with $a$ being $a/2$. This corollary will be important to control $\Delta_3$ of Proposition \ref{prop:Deloc2}.
\begin{cor}\label{cor:markov}Fix $\Lambda$ and $\eta$.
Let $s\in\Omega^{\Lambda}$. If $\eta^s$ and $\eta^{-s}$ have the same distribution then for any $a>2$
\begin{align*}
\P \big( (\varphi^{\eta^{-s}}-\varphi^{\eta^s},-\Delta s\,) & \le (a-2)\|\nabla s\|^{2}\, \big)\ge1-\frac{2}{a}.
\end{align*}
    
\end{cor}

We will need the following noise decomposition that will be a key to understand the impact of resampling $\hat\eta_\Lambda$ (from Theorem \ref{theorem:83-1}) on the noise. In particular, we can find two levels $\beta\ge \alpha$ so that we know that the impact of resampling $\hat\eta_\Lambda$ is more important for points that have height larger than $\beta h $ compared to points that have height smaller than $h$. 
\begin{lem}[Noise decomposition]\label{lem:noise}Let $H\in(0,1)$. There exist $\beta\ge \alpha\ge 1$ such that the following holds. Let
$h\ge 1$ and  $\Lambda'\subset\Lambda$. 
Set 
\[
\hat\eta:=\sum_{v\in\Lambda'}\eta_{v,\alpha h e_1}+\eta_{v,-\alpha h e_1}.
\]
For all $v\in\Lambda$ and $t\in\R^{n}$ denote $\kappa_{v,t}\coloneqq\frac{\cov(\eta_{v,t},\hat \eta)}{\var(\hat \eta)}$
and define 
\[
\eta_{v,t}^{\perp}:=\eta_{v,t}-\kappa_{v,t}\hat \eta.
\]
Then, $\hat \eta$ and $(\eta_{v,t}^{\perp})_{v\in\Lambda,t\in\R^{n}}$
are independent, for $v\notin\Lambda'$, $  \Cov(\hat \eta, \eta_{v,t})=0$ and for any $v\in\Lambda'$ 
\begin{equation}\label{eq:lemnoise1}
 \forall t \in \R^n\qquad 0\le   \Cov(\hat \eta, \eta_{v,t})\le 4(\alpha h)^{2H},
\end{equation}
\begin{equation}\label{eq:lemnoise1bis}
    \sup_{\|t\|\le  h }\Cov(\hat \eta, \eta_{v,t})\le \frac 18 (\alpha h)^{2H}
\end{equation}
and
\begin{equation}\label{eq:lemnoise2}
    \inf_{\|t\|\ge \beta h }\Cov(\hat \eta, \eta_{v,t}) \ge  \frac 12 (\alpha h)^{2H}.
\end{equation} 
\end{lem}

For the proof of Lemma~\ref{lem:noise} we will need the next claim. The proof of the claim is postponed to the appendix.

\begin{claim}\label{claim:estimate} Let $H\in(0,1)$. For every $t\in\R^n$
\begin{equation}\label{eq:claim 1} 2\|t\|^{2H}+2-\|t-e_1\|^{2H}-\|t+e_1\|^{2H}\le \min(2(\|t\|^{2H}+\|t\|),8)
       \end{equation}
       and
       \begin{equation}\label{eq:claim 2}
          2\|t\|^{2H}+2-\|t-e_1\|^{2H}-\|t+e_1\|^{2H}\ge  2g(\|t\|)\ge 0 
       \end{equation}
       where $g(x)\coloneqq 2x^{2H}+2- 2(x^2+1)^H$.
\end{claim}

\begin{proof}[Proof of Proposition \ref{prop:close_surf_markov}] Denote $\eta(\varphi)\coloneqq\sum_{v\in\Lambda}\eta_{v,\varphi_{v}}$
so that $H^{\eta}(\varphi)=\frac{1}{2}\|\nabla\varphi\|^{2}+\eta(\varphi)$
and more generally $H_{s}^{\eta}(\varphi)=\frac{1}{2}\|\nabla(\varphi-s)\|^{2}+\eta(\varphi)-\eta(s)$.
Hence 
\[H^\eta _s(\varphi)=\frac 12 \|\nabla \varphi\|^ 2 +\frac 12 \|\nabla s\|^ 2+(\varphi,\Delta s)+\eta(\varphi)-\eta(s)= H^\eta(\varphi)+\frac 12 \|\nabla s\|^ 2+(\varphi,\Delta s)-\eta(s).\]
This yields
\begin{align*}
0 & \le H^{\eta}(\varphi_{s}^{\eta})-H^{\eta}(\varphi^{\eta})\\
 & =H_{s}^{\eta}(\varphi_{s}^{\eta})-H_{s}^{\eta}(\varphi^{\eta})+(\varphi_{s}^{\eta}-\varphi^{\eta},-\Delta s\,)\\
 & \le0+D
\end{align*}
where we used in the last inequality that $\varphi^ \eta_s$ is the minimizer of $H_{s}^{\eta}$. Since $\eta^s$ and $\eta$ have the same distribution, it yields that 
\[\E(\varphi^ \eta,-\Delta s)= \E(\varphi^ {\eta^ {-s}},-\Delta s)\]
and
\[\E D= (s,-\Delta s)=\|\nabla s \|^ 2.\]
The result follows by applying Markov's inequality.
\end{proof}

\begin{proof}[Proof of Lemma \ref{lem:noise}]
We have for $v\in\Lambda$ and $t\in\R^n$
\[\Cov(\eta_{v,t}^{\perp},\hat \eta)= \Cov(\eta_{v,t},\hat \eta)-\kappa_{v,t}\var(\hat\eta)=0.\]
Then, $\hat \eta$ and $(\eta_{v,t}^{\perp})_{v\in\Lambda,t\in\R^{n}}$
are independent.
Let $v\in\Lambda'$, we have for $t\in\R^n$
\begin{equation}\label{eq:computation covariance}
\begin{split}
    \Cov(\hat \eta, \eta_{v,t})&=\Cov(\eta_{v,\alpha h e_1},\eta_{v,t})+\Cov(\eta_{v,-\alpha h e_1},\eta_{v,t})\\
    &=
    \frac 12(2\|t\|^{2H}+2(\alpha h )^{2H}-\|t-\alpha h  e_1\|^{2H}-\|t+\alpha h  e_1\|^{2H})\\
    &=
    \frac {(\alpha h)^{2H}}2\left(2\left\|\frac t{\alpha h}\right\|^{2H}+2-\left\|\frac t{\alpha h}-e_1\right\|^{2H}-\left\|\frac t{\alpha h}+ e_1\right\|^{2H}\right)\\
      &\le
   \frac {(\alpha h)^{2H}}2\min\big(2\big(\frac {\|t\|^{2H} }{ (\alpha h)^{2H}}+ \frac {\|t\| }{ \alpha h}\big),8\big)=\min( \|t\|^{2H} +(\alpha h)^{2H-1}\|t\|, 4(\alpha h)^{2H})
    \end{split}
\end{equation}
where we used \eqref{eq:claim 1} of Claim \ref{claim:estimate} in the last inequality. Thanks to \eqref{eq:claim 2}, we deduce that the covariance is non-negative yielding \eqref{eq:lemnoise1}.
We now choose $\alpha$ large enough such that
\[1+\alpha^{2H-1}\le \frac 18 \alpha^{2H}.\]
Inequality \eqref{eq:lemnoise1bis} follows from the previous inequality together with \eqref{eq:computation covariance}.

We have for $\|t\|\ge \beta h$ using \eqref{eq:claim 2}
    \begin{equation}\label{eq:lowerboundcov2}
    \begin{split}
        \Cov(\hat \eta, \eta_{v,t})&=\frac {(\alpha h)^{2H}}2\left(2\left\|\frac t{\alpha h}\right\|^{2H}+2-\left\|\frac t{\alpha h}-e_1\right\|^{2H}-\left\|\frac t{\alpha h}+ e_1\right\|^{2H}\right)\\&\ge (\alpha h)^{2H}g(\|t\|/\alpha h)\ge (\alpha h)^{2H}g(\beta/\alpha)\ge\frac 12\alpha^{2H}h^{2H}
        \end{split}
    \end{equation}
    where in the last inequality we pick some $\beta$ large enough depending on $H$ and $\alpha$ such that $g(\beta/\alpha)\ge 1/2$ (it is easy to check that $\lim _{x\rightarrow\infty} g(x)=1$). Recall that $g:x\mapsto x^H+1-(x+1)^H$ was introduced in the proof of Claim \ref{claim:estimate}.
The result follows.
\end{proof}

\subsection{Proof of Theorem \ref{theorem:83-1}}\label{sec:proof}

We first give a high-level description of the proof. We start with the disorder $\eta$ and consider a coupled disorder $\zeta$ obtained by replacing $\hat\eta_\Lambda$ with an independent copy $\hat\zeta_\Lambda$, while fixing all the orthogonal degrees of freedom. On the event $\{\hat\eta_\Lambda\le -Ch^2\ell ^{d-2}\}$, with probability at least $1/2$, we have that $\hat\zeta_\Lambda- \hat\eta_\Lambda\ge Ch^2\ell^{d-2}$. Next, let $\Pi$ be the set of surfaces whose height is at most $h$ at most vertices. If we consider a shift supported on \(\Lambda\) and equal to \(\beta h e_1\) in the bulk of \(\Lambda\), we can show, based on our choice of \(\beta\) and \(\alpha\) from Lemma \ref{lem:noise}, that the effect of resampling \(\hat\eta_\Lambda\) is greater for surfaces in \((\Pi + s) \cup (\Pi - s)\) compared to those in \(\Pi\). This leads to the conclusion that the observable \(\Delta_1 + \Delta_2\) (from Proposition \ref{prop:Deloc2}) is at least of the order \(\hat\zeta_\Lambda - \hat\eta_\Lambda \ge C h^2\ell^{d-2}\). The observable \(\Delta_3\) is then controlled using the Markov inequality and is of order \(\|\nabla s\|^2 = O(h^2\ell^{d-2})\) (see Propositions \ref{prop:close_surf_markov} and \ref{prop:const_laplace}). Hence, we conclude that on the event \(\{\hat\eta_\Lambda \le -Ch^2\ell^{d-2}\}\), with positive probability, at least one of the surfaces \(\varphi^{\zeta^s}\), \(\varphi^{\zeta^{-s}}\), or \(\varphi^\eta\) is not in \(\Pi\). To upper bound the probability that $\varphi^{\zeta^s}\notin\Pi $ by the probability that $\varphi^\eta\notin \Pi$ requires taking a slightly more involved definition of $\Pi $.

We proceed to describe the proof in detail.

\begin{proof}[Proof of Theorem \ref{theorem:83-1}] Let $\beta\ge \alpha$ be the values from Lemma \ref{lem:noise}.
Consider the noise $\zeta$ that corresponds to the noise $\eta$ where we resample the value of $\hat \eta_\Lambda$ by a new value $\hat \zeta_\Lambda$. Specifically, we have thanks to Lemma \ref{lem:noise}
\begin{equation}\label{eq:etaequality}
\zeta_{v,t}:=\eta_{v,t}+\kappa_{v,t}(\hat\zeta_\Lambda-\hat{\eta}_\Lambda).  
\end{equation}
For short, we write $\hat\eta_\Lambda=\hat \eta$ and $\overline \P$ for the probability conditionally on $\hat{\eta}_{\Lambda}$ and $\mathcal{F}_{\Lambda^{c}}$. 

Thanks to Lemma \ref{lem:noise}, we have
\begin{equation}\label{eq:covbound}
\begin{split}
    \Cov(\eta_{v,t},\hat{\eta})\ge\frac 12 (\alpha h)^{2H} \,\text{ for }\|t\|\ge \beta h\\
\Cov(\eta_{v,t},\hat{\eta})\le\frac 18 (\alpha h)^{2H} \,\text{ for }\|t\|\le  h\\
0\le \Cov(\eta_{v,t},\hat{\eta})\le4 (\alpha h)^{2H} \,\text{ for }t\in\R^n.
\end{split}
\end{equation}
Define 
\begin{equation*}
    \Pi  \coloneqq\Big\{\varphi\in\Omega^{\Lambda_{L}}:\sum_{v\in\Lambda}\Cov(\eta_{v,\varphi_v},\hat\eta )\le \frac 1 {6}|\Lambda|  (\alpha h)^{2H}\Big\}.
\end{equation*}
By \eqref{eq:covbound}, we have
\begin{equation}\label{eq:boundphih}
  \forall\varphi\in\Pi,\quad  \big|\big\{ v \in \Lambda  : \| \varphi _v \| \ge \beta  h \big\}\big|\le \frac 13 |\Lambda|.
\end{equation}
Let $\varphi \notin\Pi $ and set \[D_\Lambda(\varphi)\coloneqq |\{v\in\Lambda:\|\varphi_v\|>h\} |.\]  Inequality \eqref{eq:covbound} implies that
\begin{equation*}
  \frac 1 {6}|\Lambda|  (\alpha h)^{2H}<  \sum_{v\in\Lambda}\Cov(\eta_{v,\varphi_v},\hat\eta )\le \frac 18( |\Lambda|-D_\Lambda(\varphi) ) (\alpha h)^{2H}+ 4(\alpha h)^{2H}D_\Lambda(\varphi)
\end{equation*}
where the lower bound follows from the fact that $\varphi\notin\Pi$.
Yielding 
\[D_\Lambda(\varphi)\ge \frac {1}{93}|\Lambda|.\]
Hence, the result will follow by proving that $\overline \P(\varphi ^\eta\notin\Pi)\ge c$ on the event $\{\hat{\eta}_{\Lambda}\le - Ch^2\ell ^{d-2}\}$.

Let $\Delta_1,\Delta_2,\Delta_3$ be as in Proposition \ref{prop:Deloc2}.
We start by bounding $\Delta_{1}$. Using \eqref{eq:etaequality}, on the event $\{ \hat \zeta\ge \hat \eta\}$ we have 
\begin{equation}\label{eq:bounddelta1}
\begin{split}
\Delta_{1} =\GE_{\Pi}^{\eta}-\GE_{\Pi}^{\zeta}\ge\inf_{\varphi\in\Pi}H^{\eta}(\varphi)-H^{\zeta}(\varphi)&=\inf_{\varphi\in\Pi}\sum_{v\in\Lambda} \eta_{v,\varphi_v}-\zeta_{v,\varphi_v}  =\inf_{\varphi\in\Pi}\sum_{v\in\Lambda}\kappa_{v,\varphi_v }(\hat \eta-\hat \zeta)\\&\ge -\frac{1}{6\var(\hat \zeta)} (\alpha h)^{2H}|\Lambda|(\hat{\zeta}-\hat\eta ).\end{split}
\end{equation}
Let $\pi$ be the function from Proposition \ref{prop:const_laplace} applied to the box $\Lambda$ and $\ep=1/10$. Recall that $\Lambda^-$ was defined in the statement of Proposition \ref{prop:const_laplace}.
Define the surface $s\coloneqq 2\beta \pi e_1$ where $\beta$ was defined in Lemma \ref{lem:noise}. In particular, we have by inequality \eqref{eq:covbound}
\begin{equation}\label{eq:boundsumkappa}
    \forall v\in\Lambda^-\quad\forall \|t\|\le \beta h \qquad \min(\Cov(\eta_{v,t+s_v},\hat\eta ),\Cov(\eta_{v,t-s_v},\hat\eta ))\ge \frac 12 (\alpha h)^{2H}
\end{equation}
Let $\varphi\in(\Pi+s)\cup(\Pi-s)$. Thanks to \eqref{eq:boundphih} there is at least $\frac 23|\Lambda|$ vertices such that $\|\varphi_v\|\le \beta h$. Since $|\Lambda\setminus\Lambda^-|\le \ep|\Lambda|$, thanks to \eqref{eq:boundsumkappa}, there is at least $(\frac 23-\ep)|\Lambda|=\frac {19}{30}|\Lambda|$ vertices in $\Lambda^-$ such that $\Cov(\eta_{v,\varphi_v},\hat \zeta)\ge \frac 12 (\alpha h)^{2H}$.
It yields as in \eqref{eq:bounddelta1} that on the event $\{\hat \zeta\ge \hat\eta\}$,
\begin{equation}\label{eq:bounddelta2}
\begin{split}
\Delta_{2}  =\inf_{\varphi\in (\Pi+s)\cup(\Pi-s)}H^\zeta(\varphi)-H^\eta(\varphi)
&\ge \frac{19}{60\var(\hat \zeta)} |\Lambda|(\alpha h)^{2H}(\hat{\zeta}-\hat\eta).
\end{split}
\end{equation}
Moreover, we have
\begin{equation}\label{eq:xi positive}
 \overline \P(\hat{\zeta}\ge0)\ge\frac 12.
\end{equation}
Thus combining  \eqref{eq:bounddelta1}, \eqref{eq:bounddelta2} and \eqref{eq:xi positive} yields that on the event $\{\hat{\eta}_{\Lambda}\le - Ch^2\ell ^{d-2}\}$
\[
\overline \P\bigg( \Delta_{1}+\Delta_{2}\ge\frac {3}{20\var(\hat \zeta)} |\Lambda|(\alpha h)^{2H}(\hat{\zeta}-\hat \eta )\bigg) \ge\frac{1}{2}.
\]

By Corollary \ref{cor:markov} for $a=6$, using that $\zeta ^s$ and  $\zeta ^{-s}$ have the same distribution under the law $\overline \P$, we get
\[
\overline \P(\,\Delta_{3}\ge-2\|\nabla s\|^{2}\,)\ge\frac{2}{3}
\]
If the following holds
\begin{equation}\label{eq:to beverify}
    \frac {3}{20\var(\hat \zeta)} |\Lambda|(\alpha h)^{2H}Ch^2\ell ^{d-2}\ge3\|\nabla s\|^{2}
\end{equation}
then by Proposition
\ref{prop:Deloc2}, it holds that on the event $\{\hat{\eta}_{\Lambda}\le - Ch^2\ell ^{d-2}\}$
\begin{equation}\label{eq:deloc4dl1}
\begin{split}
     \overline \P(\varphi^\eta\notin\Pi)+\overline \P(\varphi^{\zeta^{s}}\notin\Pi)+\overline \P(\varphi^{\zeta^{-s}}\notin\Pi)&\ge \overline \P(\Delta _1+\Delta _2+\Delta _3\ge \|\nabla s\|^{2})\\
     &\ge \overline \P( \Delta_{1}+\Delta_{2}\ge 3\|\nabla s\|^{2},\Delta_{3}\ge-2\|\nabla s\|^{2})\\ 
    & \ge \frac{1}{2}+\frac{2}{3}-1=\frac{1}{6}. 
\end{split}
\end{equation}
We will fix the constants at the end such that \eqref{eq:to beverify} holds.

We claim that 
\begin{equation}\label{eq:claim monotony}
 \{ \varphi^\zeta\notin\Pi, \ \hat\eta \le \hat \zeta \}\subset \{\varphi^\eta\notin\Pi\}.
\end{equation}
Assume $\hat \zeta\ge \hat\eta $ and $\varphi^\zeta\notin \Pi$.
We have
\begin{equation*}
    \begin{split}
        \GE^\zeta_\Pi-\GE^\eta_\Pi&\le H^{\zeta}(\varphi^{\eta,\Pi})- H^{\eta}(\varphi^{\eta,\Pi})\le\sup_{\varphi\in \Pi} H^{\zeta}(\varphi)- H^{\eta}(\varphi) \\&=\sup_{\varphi\in \Pi}\frac{1}{\var(\hat \zeta)}\sum_{v\in\Lambda}\Cov(\eta_{v,\varphi_v},\hat\eta)(\hat\zeta -\hat\eta )\le \frac{1}{6\var(\hat \zeta)} (\alpha h)^{2H}|\Lambda|(\hat\zeta -\hat\eta )
    \end{split}
\end{equation*}
and similarly
\begin{equation*}
    \begin{split}
        \GE^\eta-\GE^\zeta&\le H^\eta(\varphi^\zeta)-H^\zeta (\varphi^\zeta)\le \sup_{\varphi\notin\Pi} H^\eta(\varphi)-H^\zeta (\varphi)\\
        &\le\frac{1}{\var(\hat \zeta)}(\hat\eta -\hat\zeta)\inf_{\varphi\notin\Pi}\sum_{v\in\Lambda}\Cov(\eta_{v,\varphi_v},\hat\eta)< -\frac{1}{6\var(\hat \zeta)} (\alpha h)^{2H}|\Lambda|(\hat\zeta -\hat\eta )
    \end{split}
\end{equation*}
where we used that $\varphi^\zeta\notin \Pi $. By summing the two inequalities and using again that $\varphi^\zeta\notin \Pi $, we get 
\[\GE^\eta-\GE^\eta_\Pi\le \GE^\zeta-\GE^\zeta_\Pi<0.\]
It yields that $\varphi ^\eta\notin\Pi$ and \eqref{eq:claim monotony} follows.

Using that $\zeta^{-s}, \zeta$ and $\zeta^{s}$ have the same distribution under the law $\overline \P$, we get
\[\overline \P(\varphi^{\zeta^{s}}\notin\Pi)+\overline \P(\varphi^{\zeta^{-s}}\notin\Pi)\le 2\overline \P(\varphi^{\zeta}\notin\Pi)\le 2\overline \P(\varphi^{\eta}\notin\Pi)+2\overline \P( \hat \zeta<\hat \eta )\]
where the last inequality follows from \eqref{eq:claim monotony}.
Combining the previous inequality together with inequality \eqref{eq:deloc4dl1} gives
\[2\overline \P( \hat \zeta<\hat \eta )+3\overline \P(\varphi^\eta\notin\Pi)\ge \frac{1}{6}.\]
We will further choose $C$ large enough such that 
\begin{equation}\label{eq:hatzeta>a}
    \overline \P( \hat \zeta< - Ch^2\ell ^{d-2})\le \frac 1 {24}.
\end{equation}
It yields that on the event $\{\hat{\eta}_{\Lambda}\le - Ch^2\ell ^{d-2}\}$
\[\P(\varphi^\eta\notin\Pi \mid \hat \eta_\Lambda, \mathcal F_{\Lambda^c})\ge \frac{1}{36}.\]
To conclude we have to check that inequalities \eqref{eq:to beverify} and \eqref{eq:hatzeta>a} hold.
Note that $\var(\hat \zeta)=C_0\ell^d h^ {2H}$,
$\frac {1}{4\var(\hat \zeta)} |\Lambda|(\alpha h)^{2H}C\ell ^{d-2}h^2= \Omega(C\ell ^{d-2}h^2)$ and
$\|\nabla s\|^{2}=O (h^2\ell^{d-2})$ so the result follows by choosing $C$ large enough depending only on $H$ and $d$. To conclude that \eqref{eq:hatzeta>a} holds note that
\[ \frac { C\ell ^{d-2}h^{2}} {C_0\ell ^{d/2}h^{H}}=\frac{C}{C_0}h^{2-H}\ge \frac{C}{C_0}.\]
Hence \eqref{eq:hatzeta>a} holds for $C$ large enough.
\end{proof}

\subsection{Lower bounds on height fluctuations}\label{sec:dneq4}

The lower bounds in dimensions $d\ne 4$ follow directly from Theorem \ref{theorem:83-1}. The case of dimension $d=4$ is more complicated, and is proved with a fractal percolation construction.

\begin{proof}[Proof of~\eqref{eq:constant fraction delocalization} and the lower bound in \eqref{eq:123 height fluctuations thm} (Theorem \ref{thm:main 123})]
In dimensions $d\le 3$, apply Theorem \ref{theorem:83-1} to $\Lambda=\Lambda_L$ and $h=tL^{\frac{4-d}{4-2H}}$.
We have
\[\var(\hat\eta)=  |\Lambda_{L}| (4-2^{2H})(\alpha h )^{2H} \]
and thanks to our choice of $h$,
\begin{equation}\label{eq:choice of h}
    \frac{h^2L^{d-2}}{\sqrt{\var(\hat \eta)}}=O(h^{2-H}L^{\frac d 2 -2})=O(t^{2-H}) .
\end{equation}
Set 
\[\Pi\coloneqq\bigg\{\varphi\in\Omega^{\Lambda}:|\big\{ v \in \Lambda  : \| \varphi _v^\eta  \| >  h \big\}\big|\ge \frac c2|\Lambda|\bigg\} \]
where $c$ is the constant from Theorem \ref{theorem:83-1}.
Thanks to Theorem \ref{theorem:83-1}, we have that conditionally on $\{\hat \eta\le -Ch^2L^{d-2}\}$
\begin{equation}\label{eq:lowerbound proba Pi}
|\Lambda|\P(\varphi^{\eta}\in\Pi\mid\hat\eta)+\frac c2|\Lambda|\ge \mathbb{E}\big[\big|\big\{ v \in \Lambda  : \| \varphi _v^{\eta}  \| >  h \big\}\big|\mid\hat{\eta}\big]\ge c|\Lambda|.
\end{equation}
Yielding
\begin{equation}\label{eq:density delocalization}
    \mathbb P \left(\frac{1}{|\Lambda_L|} \big| \big\{v\in\Lambda_L:\|\varphi_v^{\eta,\Lambda_L}\|\ge tL^{\frac{4-d}{4-2H}} \big\} \big| \ge \frac c2\right)=\P(\varphi^\eta\in\Pi)\ge \frac c2 \P(\hat \eta\le -Ch^2L^{d-2})\ge \frac c2 e^{-ct^{2-H}}
\end{equation}
where we used \eqref{eq:choice of h} in the last inequality. This proves~\eqref{eq:constant fraction delocalization}, which implies the lower bound in~\eqref{eq:123 height fluctuations thm}.
\end{proof}

\begin{proof}[Proof of the lower bound in \eqref{eq:loweboundheight d>4} (Theorem \ref{thm:main d>4})]
The result follows readily by applying Theorem \ref{theorem:83-1} to $\Lambda=\{v\}$ and $h=t$ for every $v\in \Lambda$.
\end{proof}

We turn to prove a lower bound for the typical heights in dimension $d=4$. The proof is a simple application of fractal (Mandelbrot) percolation using Theorem \ref{theorem:83-1} (related arguments were used in dimension $d=4$ in~\cite{dario2024quantitative,dembin2024minimal}).

\begin{proof}[Proof of \eqref{eq:fluctuation height d=4} (Theorem~\ref{thm: main 4})] 
Suppose that $L$ is sufficiently large. We would like to apply Theorem~\ref{theorem:83-1} with $h:=c_0  (\log \log L)^{\frac{1}{4-2H}}$ for a sufficiently small constant $c_0>0$ to be determined later. 
For a box $\Lambda \subseteq \Lambda _L$, let $ D_{\Lambda }$ be the number of vertices $v\in \Lambda $ on which the surface delocalizes to height $ h$. That is, 
\begin{equation}
     D _{\Lambda  } := \big|\big\{ v \in \Lambda  : \| \varphi _v^\eta \| >  h \big\}\big| .
\end{equation}
 We define $\hat \eta _\Lambda $ to be 
\begin{equation}
    \hat \eta_\Lambda :=\sum _{v\in \Lambda} \eta_{v,\alpha h e_1}+\eta_{v,-\alpha h e_1}.
\end{equation}
We say that the block $\Lambda=v+ \Lambda _\ell  \subseteq \Lambda _L$ is \emph{good} if $\hat \eta _\Lambda \le -a(\ell) $ with $a(\ell) =C\ell ^2 h^{2}$ as in Theorem \ref{theorem:83-1}. Note that $\hat\eta _\Lambda$ is normally distributed with zero expectation and $\var(\hat\eta _\Lambda)\ge c\ell ^4h^{2H} \ge 10a(\ell )^2/\log \log L$, as long as $c_0$ is sufficiently small depending on $C$, and therefore $\mathbb P (\Lambda \text{ is good}) \ge (\log L)^{-1/3}$ for large $L$.

Our goal will be to construct a set $\mathcal Q$ of disjoint good blocks that cover a positive fraction of $\Lambda _L$ with high probability. To this end, we consider a sequence of hierarchical partitions. Let $j_1:=\lfloor \tfrac{1}{2} \sqrt{\log L} \rfloor $ and for any integer $j\in[0,j_1]$ let $m_j:=\lfloor \log _3L \rfloor -j\lfloor \sqrt{\log L}\rfloor  $ and $R_j:=3^{m_j}$. Define $\Lambda ^0:=[0,R_0)^d \subseteq \Lambda _L$ and let
\begin{equation}
\mathcal T _j:= \Big\{  \Lambda =\big( z+[0,R_j)^d \big) \cap \mathbb Z ^d : z\in R_j \mathbb Z ^d \text{ and } \Lambda \subseteq \Lambda ^0 \Big\}.
\end{equation} 
Note that $\mathcal T _0=\{\Lambda ^0\}$ and that $\mathcal T _{j+1}$ is a refinement of the partition $\mathcal T _j$ such that any block $\Lambda \in \mathcal T _j$ is the union of $3^{d\lfloor \sqrt{\log L} \rfloor }$ blocks in $\mathcal T _{j+1}$. Moreover, by construction, the side length of each block in $\mathcal T _j$ is odd and therefore it is of the form $v+\Lambda _{\ell _j} $ with $\ell _j:=(R_j-1)/2$.

We define $\mathcal Q$ in the following way. A block $\Lambda \in \mathcal T _j$ belongs to $\mathcal Q$ if and only if $\Lambda $ is good and for all $j'<j$ the unique block $\Lambda ' \in \mathcal T _{j'}$ with $\Lambda \subseteq \Lambda '$ is not good. We say that a vertex $v\in \Lambda _L$ is covered if $v\in \Lambda $ for some $\Lambda \in \mathcal Q $. We claim that for all $v\in \Lambda ^0$
\begin{equation}\label{eq:bound for v}
    \mathbb P (v\text{ is covered}) \ge 1/2.
\end{equation}
Indeed, letting $\Lambda ^j=\Lambda ^j(v) $ be the unique block in $\mathcal T _j$ containing $v$ we have
\begin{equation}\label{eq:44}
    \Big\{ \!\! \begin{array}{c} v \text{ is not} \\ \text{covered}\end{array} \!\! \Big\} \subseteq \bigcap _{j=0}^{j_1-1} \!\big\{ \hat \eta _{\Lambda ^j} \ge-  a(\ell _j) \big\} \subseteq \bigcup _{j=0}^{j_1-1} \! \big\{\hat \eta _{\Lambda ^{j+1}} \ge  a(\ell _j) \big\} \cup \bigcap _{j=0}^{j_1-1} \!\big\{ \hat \eta _{\Lambda ^j}-\hat \eta _{\Lambda ^{j+1}} \ge - 2a(\ell _j) \big\}.
\end{equation}
The random variable $\hat \eta  _{\Lambda ^{j+1}}$ is normal with $\var (\hat \eta  _{\Lambda ^{j+1}})\le C\ell _{j+1}^{4}h^{2H} \le a(\ell _j)^2\exp (-\sqrt{\log L})$ and therefore $\mathbb P(\hat \eta  _{\Lambda ^{j+1}} \ge a(\ell _j)) \ll 1/L $. Moreover, the variables $\{\hat \eta  _{\Lambda ^j}-\hat \eta  _{\Lambda ^{j+1}}\}_{j=0}^{j_1-1}$ are independent and normal with $\var (\hat \eta  _{\Lambda ^j}-\hat \eta  _{\Lambda ^{j+1}})\ge c\ell _j^{4}h^{2H}\ge 10a(\ell _j)^2/\log \log L$, as long as $c_0$ is sufficiently small, and therefore the probability of the intersection on the right hand side of \eqref{eq:44} is at most $(1-(\log L)^{-1/3})^{j_1} \le 1/4$. Substituting these bounds in \eqref{eq:44} completes the proof of \eqref{eq:bound for v}.

Next, note that the for all $j'\le j$ and any $\Lambda \subseteq \Lambda '$ with $\Lambda \in \mathcal T_j$ and $\Lambda '\in \mathcal T _{j'}$, the random variable $\hat \eta _{\Lambda '}$ is measurable in $ \sigma (\hat \eta _\Lambda , \mathcal F _{\Lambda ^c} )$. Thus, the event $\{\Lambda \in \mathcal Q \}$ is measurable in $ \sigma (\hat \eta _\Lambda , \mathcal F _{\Lambda ^c} )$ and therefore by Theorem \ref{theorem:83-1} we have 
\begin{equation}
    \mathbb E \big[ 
D_{\Lambda }\mathds 1 \{\Lambda \in \mathcal Q \} \big] =\mathbb E \big[ \mathds 1\{\Lambda \in \mathcal Q \} \cdot \mathbb E [D_\Lambda \mid \hat \eta _\Lambda , \mathcal F _{\Lambda ^c} ]  \big]\ge c|\Lambda | \cdot  \mathbb P (\Lambda \in \mathcal Q ).
\end{equation}
Hence, we obtain
    \begin{equation}
    \begin{split}
        \mathbb E [D_{\Lambda _L}] \ge \mathbb E \Big[ \sum _{\Lambda \in \mathcal Q } D_\Lambda \Big]\ge \sum _{j=0}^{j_1} \sum _{\Lambda \in \mathcal T _j} \mathbb E \big[ 
D_{\Lambda }\mathds 1 \{\Lambda \in \mathcal Q \} \big]\ge c \sum _{j=0}^{j_1} \sum _{\Lambda \in \mathcal T _j} |\Lambda |\cdot \mathbb P (\Lambda \in \mathcal Q \big)\\
=c \cdot \mathbb E \Big[ \sum _{\Lambda \in \mathcal Q } |\Lambda | \Big] =c \sum _{v\in \Lambda ^0 } \mathbb P \big( v \text{ is covered} \big) \ge \frac{c}{2} |\Lambda ^0| \ge c|\Lambda _L|,
\end{split}
\end{equation}
where in the second to last inequality we used \eqref{eq:bound for v}. This finishes the proof.
\end{proof}

\subsection{Ground energy anti-concentration}\label{sec:G} In this section we prove the lower bounds on the ground energy fluctuations stated in our main results; specifically, the lower bounds in~\eqref{eq:bound height 123},~\eqref{eq:lowerbound26} (low dimensions $d<4$),~\eqref{eq:27} (critical dimension $d=4$) and~\eqref{eq:GE d>4} (high dimensions $d>4$).

The lower bounds in~\eqref{eq:27} and~\eqref{eq:GE d>4} are deduced from the following general result, which derives a lower bound on the ground energy fluctuations from a lower bound on the typical heights. This inequality may be thought of as the direction
\begin{equation}
    \chi \ge H\xi + \frac{d}{2}
\end{equation}
of the scaling relation~\eqref{eq:second scaling relation}. In its proof, starting from the assumption that a positive fraction of the heights exceed $h$ with uniformly positive probability, fluctuations in the ground energy are shown to follow from fluctuations of the disorder at the heights $he_1$ and $-he_1$.

\begin{theorem}\label{thm:GEAC} Let $H\in(0,1)$ and integer $d,n\ge 1$.
For each $\delta>0$ there exists $c_\delta>0$, depending only on $\delta,H, d$, such that the following holds. For each finite $\Lambda^0\subset\Lambda^1\subset\Z^d$, if $h>0$ satisfies 
\begin{equation}\label{eq:assumption expected fraction of delocalizatio}
    \frac{\mathbb E \big[ | \{v\in \Lambda^0 : \|\varphi^{\eta,\Lambda^1}_v\| \ge h \} | \big]}{|\Lambda^0|} \ge \delta
\end{equation}
then
\begin{equation}
    \P\left(\left|\GE^{\eta,\Lambda^1}-\med(\GE^{\eta,\Lambda^1})\right|\ge c_\delta |\Lambda^0|^{1/2}h^H\right)\ge c_\delta.
\end{equation}
\end{theorem}

The lower bound~\eqref{eq:27} (respectively~\eqref{eq:GE d>4}) for dimensions $d=4$ (respectively $d>4$) directly follows from Theorem~\ref{thm:GEAC} with the choice of parameters $\Lambda^0=\Lambda^1=\Lambda_L$ and   $h=c(\log\log L)^{\frac 1{4-2H}}$) (respectively $h=1$) using the lower bound in~\eqref{eq:fluctuation height d=4} (respectively  \eqref{eq:loweboundheight d>4}).

With the same approach in dimensions $d<4$, the lower bound in~\eqref{eq:bound height 123} follows from the lower bound in~\eqref{eq:constant fraction delocalization}. However, we use a different approach to prove the lower bound in~\eqref{eq:lowerbound26} (from which the lower bound in~\eqref{eq:bound height 123} also follows), deriving it in the following argument using Theorem \ref{theorem:83-1}.

\begin{proof}[Proof of the lower bound in \eqref{eq:lowerbound26} (Theorem \ref{thm:main 123})]
Let $d\in\{1,2,3\}$.
Let $L\ge 3$, $h\ge 1$. Set $\Lambda=\Lambda_L$. Let $\alpha$ from  Lemma \ref{lem:noise}. Set 
\begin{equation}
    \hat \eta_\Lambda :=\sum _{v\in \Lambda} \eta_{v,\alpha h e_1}+\eta_{v,-\alpha h e_1}.
\end{equation}
We will now use the noise decomposition as the one introduced in Lemma \ref{lem:noise}.
For all $v\in\Lambda$ and $t\in\R^{n}$ denote $\kappa_{v,t}\coloneqq\frac{\cov(\eta_{v,t},\hat \eta)}{\var(\hat \eta)}$
and define 
\[
\eta_{v,t}^{\perp}:=\eta_{v,t}-\kappa_{v,t}\hat \eta.
\]

Let $\hat \eta^1$ and $\hat\eta^0$ be two independent random variables distributed as $\hat \eta$ and let $\eta^0$ and $\eta^1$ be the corresponding disorders, that is 
\begin{equation}\label{eq:noise12}
    \forall v,t\quad \forall i\in\{0,1\}\quad \eta_{v,t}^i=\eta_{v,t}^\perp +\kappa_{v,t}\hat \eta^i.
\end{equation}
We will work conditionally on the event $$\mathcal E\coloneqq\{\hat \eta^0\le   \hat\eta^1 - Ch^2L ^{d-2}\}\cap \{\hat{\eta}^{1}\le - Ch^2L ^{d-2}\}.$$
Set 
\[\Pi\coloneqq\bigg\{\varphi\in\Omega^{\Lambda}:|\big\{ v \in \Lambda  : \| \varphi _v^\eta  \| >  h \big\}\big|\ge \frac c2|\Lambda|\bigg\} \]
where $c$ is the constant from Theorem \ref{theorem:83-1}.
Thanks to Theorem \ref{theorem:83-1}, we have by similar computations as in \eqref{eq:lowerbound proba Pi} that conditionally on $\cE$
\[\P(\varphi^{\eta^1}\in\Pi\mid\hat\eta^1)\ge \frac c2.\]
Let $\varphi\in\Pi$.  We have thanks to \eqref{eq:noise12}
\begin{equation*}
\sum_{v\in\Lambda}\eta^0_{v,\varphi_v}=\sum_{v\in\Lambda}\kappa_{v,t}  (\hat \eta^0- \hat\eta^1)+ \sum_{v\in\Lambda}\eta^1_{v,\varphi_v}.
\end{equation*}
Moreover, using that $\varphi\in\Pi$,
\begin{equation}
  \sum_{v\in\Lambda}\kappa_{v,t}   \ge   \sum_{v\in\Lambda :\|\varphi_v\|\ge h}\frac {(\alpha h)^{2H}g(1/\alpha)}{\var(\hat \eta)}\ge c_H\frac c2 
\end{equation}
where $c_H$ only depends on $H$ (we use inequality \eqref{eq:lowerboundcov2} in the proof of Lemma \ref{lem:noise} for the last inequality) recall that $\var(\hat \eta)=O((\alpha h)^{2H}|\Lambda|)$.
Recall that $\hat\eta^0\le \hat\eta^1$ on $\cE$.
Conditionally on $\cE$,
combining the two previous inequalities, we obtain
\begin{equation*}
\sum_{v\in\Lambda}\eta^0_{v,\varphi_v}\le c_H\frac c2  (\hat \eta^0- \hat\eta^1)+ \sum_{v\in\Lambda}\eta^1_{v,\varphi_v}.
\end{equation*}
Taking the infimum for $\varphi\in\Pi$, we obtain 
\[\GE^{\eta^0,\Lambda}\le \inf_{\varphi\in\Pi}H^{\eta^0,\Lambda}(\varphi)\le c_H\frac c 2  (\hat \eta^0- \hat\eta^1)+\inf_{\varphi\in\Pi}H^{\eta^1,\Lambda}(\varphi). \]
Furthermore, conditionally on the event $\cE$, there is a probability at least $\frac c2 $ that $\varphi^{\eta^1}\in\Pi$, which imply that $\inf_{\varphi\in\Pi}H^{\eta^1,\Lambda}(\varphi)=\GE^{\eta^1,\Lambda}$. Hence,
\[\P(\GE^{\eta^0,\Lambda}-\GE^{\eta^1,\Lambda}\le -ch^2L ^{d-2})\ge \frac c 2 \P(\cE)\ge \frac c4e^{-ch^{4-2H}L^{4-d}} \]
where we used in the last inequality that $\var(\hat\eta)=O(h^2L^d)$.
Since $(\eta^0,\eta^1)$ is a coupling of the two disorders, both having the same marginal distribution as $\eta$, the conclusion follows by taking $h=\sqrt tL^{\frac{4-d}{4-2H}}$.    
\end{proof}
\begin{proof}[Proof of Theorem~\ref{thm:GEAC}]
Let $\Lambda^0\subset \Lambda^1\subset \Z^d$.
Let $h\ge 1$ and $\delta>0$ such that
\begin{equation*}
     \frac{\mathbb E \big[ | \{v\in \Lambda^0 : \|\varphi^{\eta,\Lambda^1}_v\| \ge h \} | \big]}{|\Lambda^0|} \ge \delta.
\end{equation*}
Set 
\begin{equation*}
    \Pi\coloneqq\left\{\varphi\in\Omega^{\Lambda^1}:\frac{|\{v\in\Lambda^0:\|\varphi_{v}\|\ge h\}|}{|\Lambda^0|}\ge \frac \delta 2\right\} .
\end{equation*}
In particular, by similar computations as in \eqref{eq:lowerbound proba Pi}, we get
\begin{equation}\label{eq: lower bound for phi in Pi}
    \P(\varphi ^{\eta,\Lambda^1}\in \Pi)\ge \frac \delta 2.
\end{equation}
We will now use a similar noise decomposition as the one introduced in Lemma \ref{lem:noise}.
Set 
\[
\hat\eta:=\sum_{v\in\Lambda_0}\eta_{v, h e_1}+\eta_{v,- h e_1}.
\]
For all $v\in\Lambda_1$ and $t\in\R^{n}$ denote $\kappa_{v,t}\coloneqq\frac{\cov(\eta_{v,t},\hat \eta)}{\var(\hat \eta)}$
and define 
\[
\eta_{v,t}^{\perp}:=\eta_{v,t}-\kappa_{v,t}\hat \eta.
\]
By Markov's inequality and inequality \eqref{eq: lower bound for phi in Pi}, we have
\begin{equation*}
  \P\left( \P(  \varphi ^{\eta,\Lambda^1}\notin \Pi\,|\,(\eta^\perp_{v,t})_{v\in \Lambda^1,t\in\R^n})\ge 1-\frac \delta 4\right)\le \frac {1-\frac \delta 2}{ 1-\frac \delta 4}\le 1-\frac \delta 4.
\end{equation*}
Hence, with probability at least $\delta/4$, we have that the family $(\eta^\perp_{v,t})_{v\in \Lambda^1,t\in\R^n}$ is such that 
\begin{equation*}
    \P(  \varphi ^{\eta,\Lambda^1}\in \Pi\, |\,(\eta^\perp_{v,t})_{v\in \Lambda^1,t\in\R^n})\ge \frac \delta 4.
\end{equation*}
Let us now consider a family $(\eta^\perp_{v,t})_{v\in \Lambda^1,t\in\R^n}$ such that the latter inequality holds.
Let $\hat\eta^0$ and $\hat \eta^1$ be two independent copies of $\hat \eta $.
Denote by $\eta^0$ and $\eta^1$ the two corresponding disorders, that is 
\[\forall v,t\quad \forall i\in\{0,1\},\quad \eta_{v,t}^i=\eta_{v,t}^\perp +\kappa_{v,t}\hat \eta^i. \] For short, we write $\varphi^0,\varphi^1$ for $\varphi ^{\eta^0,\Lambda^1},\varphi ^{\eta^1,\Lambda^1}$.
Thanks to the previous inequality, we get
\begin{equation*}
    \P( \varphi^0 ,\varphi^1\in \Pi\, |\,(\eta^\perp_{v,t})_{v\in \Lambda^1,t\in\R^n})\ge \frac {\delta^2} {16}.
\end{equation*}
Moreover, since $\hat \eta\sim\mathcal{N}(0,|\Lambda_0|(4-2^{2H})h^{2H})$, there exists $c_\delta>0$ depending on $H,\delta$ such that
\begin{equation*}
    \P\left(|\hat \eta^0-\hat \eta^1|\ge c_\delta |\Lambda^0|^{1/2}h^{H}\right)\ge 1- \frac {\delta^2} {32}.
\end{equation*}
By the previous inequalities, with probability at least $\delta ^3/ 2^7$, we have 
$\varphi^0 ,\varphi^1\in \Pi $ and $|\hat \eta^0-\hat \eta^1|\ge c_\delta |\Lambda^0|^{1/2}h^{H}$.
Without loss of generality, assume that $\hat \eta^0\le\hat \eta^1$. Since $\Cov(\hat \eta,\eta_{v,t})\ge (\alpha h)^{2H}g(1/\alpha)$ for $\|t\|\ge h$ (see inequality \eqref{eq:lowerboundcov2} in the proof of Lemma \ref{lem:noise}) it yields that there exists $c_H>0$ only depending on $H$ such that
\begin{equation*}
\sum_{v\in\Lambda^1}\eta^0_{v,\varphi_v^1}\le \frac \delta 2 c_H (\hat \eta^0- \hat\eta^1)+ \sum_{v\in\Lambda^1}\eta^1_{v,\varphi_v^1}
\end{equation*}
and
\begin{equation*}
\GE^{\eta^0,\Lambda^1}\le \GE^{\eta^1,\Lambda^1}-  \frac \delta2 c_H c_\delta |\Lambda^0|^{1/2}h^H.
\end{equation*}
We conclude that there exists a coupling of the disorders $\eta^0,\eta^1$ such that the marginals are distributed as $\eta $ and
\begin{equation*}
\P\left(|\GE^{\eta^0,\Lambda^1}- \GE^{\eta^1,\Lambda^1}|\ge  \frac\delta 2c_Hc_\delta |\Lambda^0|^{1/2}h^H\right)\ge \frac {\delta ^3}{2^7}.
\end{equation*}
This concludes the proof.
\end{proof}

\section{Existence, uniqueness and a priori estimates}\label{sec:verifying assumptions}

\label{sec:assumptions for eta white} In this section, we prove that the minimal surface $\varphi ^{\eta , \Lambda ,\tau }$ exists and is unique almost surely. We also obtain preliminary bounds on the transversal fluctuations and the ground energy fluctuations. In particular, we prove Proposition~\ref{prop:existence} and Proposition~\ref{as:conc}.

Using \eqref{boundary}, we observe that the existence and uniqueness of  of the minimal surface $\varphi ^{\eta ,\Lambda ,\tau}$ for general boundary values $\tau$ follow from the existence and uniqueness of the minimal surface when $\tau=0$. Once this is established for $\tau =0$, the minimal surface $\varphi ^{\eta ,\Lambda ,\tau}$ is well defined for any boundary values $\tau $ and the distributional identities in \eqref{eq:distribution of surface with general boundary values} and in Proposition~\ref{prop:effect of boundary conditions} hold. Hence, we assume throughout Sections~\ref{sec:existence} and \ref{sec:uniqueness} below that $\tau =0$.

We will use the following concentration inequality for the maximum of a Gaussian process, a consequence of the fundamental results of Borell and Tsirelson--Ibragimov--Sudakov (see, e.g., \cite[Theorem~2]{zeitouni2014gaussian}).
\begin{theorem}\label{thm: BTIS}
  Let $\{X_t\}_{t\in T}$ be a Gaussian process (not necessarily centered) on a compact space $T$ with continuous covariance function and expectation.   
 Assume $S_T\coloneqq\sup_{t\in T}X_t$ is finite almost surely. Then $\E S_T $ is finite and 
\begin{equation}\label{eq:theorem5.1}
\P(|S_T-\E S_T|\ge u )\le 2e^{-\frac {u^2}{2\sigma_T^2}}    
\end{equation}
 where $\sigma_T^2\coloneqq\sup_{t\in T} \var(X_t)$.
 \end{theorem}

\subsection{Existence and preliminary bounds}\label{sec:existence}

For brevity, we will write $\varphi^{\eta}$ for $\varphi^{\eta,\Lambda }$ and $\GE^{\eta}$ for $\GE^{\eta,\Lambda}$. Define the following sets of restricted surfaces 
\begin{equation}
\Omega_{k}:=\big\{\varphi\in\Omega^{\Lambda }:\|\nabla\varphi\|_{\Lambda}^{2}\le k|\Lambda|\big\},\quad \Omega^h\coloneqq\{\varphi \in \Omega ^{\Lambda } : \max _{v\in \Lambda } \|\varphi_v\|\le h\}.\label{eq:Omega_k}
\end{equation}
We start with the next lemma that will enable us to exclude that the minimal surface with constrained height has a too large Dirichlet energy. 
\begin{lemma}\label{lem:conc_lem} There exist $C_1,c>0$ depending
only on $d$ and $n$ such that for all $h\ge 1$ and all $k\ge C_1h^{2H}$
\begin{equation}
\mathbb{P}\Big(\exists\varphi\in\Omega_{k}\cap \Omega^h,\ \sum_{v\in\Lambda}\eta_{v,\varphi_{v}}\le-k|\Lambda|/4\Big)\le\exp\big(-ck^{2}h^{-2H}|\Lambda|\big).\label{eq:926}
\end{equation}
\end{lemma} 

Using the previous lemma, we obtain bounds on the maximal height of the minimal surface.

\begin{cor}\label{cor:maxphi}
The minimal surface $\varphi ^{\eta ,\Lambda }$ exists almost surely. Moreover, letting $L:=\max _{v\in \Lambda } \min _{u\notin \Lambda }\|u-v\|_1$, there exist $C_2,c>0$ such that for any $h\ge C_2(L|\Lambda |)^{\frac{1}{2-2H}}$, we have 
\begin{equation}\label{eq:bound135}
    \mathbb P \big( \max _{v} \|\varphi ^\eta _v\| \ge h \big) \le  \exp \Big( -c\frac{h^{4-2H}}{L^2|\Lambda |} \Big).
\end{equation}
\end{cor}

\begin{proof}
    Let $\varphi \in \Omega ^{\Lambda }$. Let $v_0$ be the vertex maximizing $\|\varphi  _v\|$ and let $v_0,v_1,\dots ,v_{L}$ be a path starting from $v_0$ and ending at $v_{L}\notin \Lambda $. Using that $\varphi _{v_{L}}=0$ and Cauchy-Schwarz inequality we obtain
\begin{equation}
    \|\varphi _{v_0} \| ^2 \le \Big( \sum _{j=1}^L \|\varphi  _{v_j}-\varphi   _{v_{j-1}}\| \Big)^2 \le L\sum _{j=0}^L \|\varphi _{v_j}-\varphi  _{v_{j-1}}\|^2 \le L \|\nabla \varphi \|^2  
\end{equation}
Thus, we have that $\Omega _{k_0(h)} \subseteq \Omega ^h$ where $k_0(h):= \lfloor  h^2/(L|\Lambda |) \rfloor $. 

Next, let $h\ge C_2(L|\Lambda |)^{\frac{1}{2-2H}}$ where $C_2$ is chosen sufficiently large so that $k_0(h) \ge C_1(2h)^{2H}$, where $C_1$ is the constant from Lemma~\ref{lem:conc_lem}. For any $\varphi \notin \Omega _h$, there is an integer $k> k_0(h)$ such that $\varphi \in \Omega _k\setminus \Omega _{k-1}$. Moreover, for such $\varphi\in \Omega _k\setminus \Omega _{k-1}$ with $k> k_0(h)$, if  $H^{\eta }(\varphi )\le h^2/(4L)$, then
\[\sum _{v\in \Lambda } \eta _{v,\varphi _v}=H^\eta (\varphi )-\tfrac{1}{2}\|\nabla \varphi \|^2 \le h^2/(4L)-(k-1)|\Lambda |/2\le -k|\Lambda |/4.\]
Thus, union bounding over integers $k>k_0(h)$ we obtain
\begin{equation}
\begin{split}\mathbb{P}\Big(\exists\varphi\in\Omega^{2h}\setminus\Omega ^h,H^{\eta}(\varphi)\le \frac{h^2}{4L}\Big)  &\le\sum_{k>k_0(h)}\mathbb{P}\Big(\exists\varphi\in\Omega_{k}\cap \Omega^{2h},\sum_{v\in\Lambda_L}\eta_{v,\varphi_{v}}\le-k|\Lambda |/4\Big)\\
 & \le\sum_{k>k_0(h)}\exp\big(-ck^2h^{-2H}|\Lambda |\big)\le\exp\Big(-c\frac{h^{4-2H}}{L^2|\Lambda |}\Big)
\end{split}
\end{equation}
where in the second inequality we used Lemma~\ref{lem:conc_lem} and that $k_0(h)\ge C_1(2h)^{2H}$ and in the last inequality we used that $h\ge C_2(L|\Lambda |)^{\frac{1}{2-2H}}\gg (L^2|\Lambda | )^{\frac{1}{4-2H}}$, so that the expression inside the exponent is large. Thus, we get
 \begin{equation*}
     \mathbb{P}\Big(\exists\varphi\in \Omega ^{\Lambda } \setminus \Omega ^{h},\ H^{\eta}(\varphi)\le \frac{h^2}{4L}\Big) \le \sum_{j\ge 1} \mathbb{P}\Big(\exists\varphi\in\Omega^{2^j h}\setminus\Omega ^{2^{j-1}h}, \ H^{\eta}(\varphi)\le \frac{h^2}{4L}\Big)  \le  \exp\Big(-c\frac{h^{4-2H}}{L^2|\Lambda |}\Big).
    \end{equation*}
Finally, note that $\Omega ^h$ is compact and that the zero function on $\Lambda $ satisfies $H^\eta(0)= 0$. Thus, on the complement of the event in the last equation, the minimal surface $\varphi ^{\eta }$ exists and is in $\Omega ^{h}$.
The bound in \eqref{eq:bound135} also follows from the last estimate.
\end{proof}

\begin{proof}[Proof of Lemma~\ref{lem:conc_lem}]
Fix $d$, $n$. Let $C_1>0$ be constant to be chosen
later. Let $k\ge C_1h^{2H}$. In order to prove \eqref{eq:926}
we discretize $\Omega_{k}$ by: 
\begin{equation}
\tilde{\Omega}_{k}:=\Omega_{k}\cap\{\varphi:\Z^{d}\to\Z^{n}\}
.\label{eq:defomegatilde}
\end{equation}
We also discretize the disorder $\eta$ in the following way. For any
$v\in\Lambda$ and $z\in\mathbb{Z}^{n}$
define 
\[
\eta'_{v,z}:=\inf\big\{\eta_{v,s}:s\in z+[0,1]^{n}\big\}.
\]
We claim that there exists $c>0$ such that for all $v\in\Lambda$, $z\in \mathbb Z ^n$ and $x>0$ we have 
\begin{equation}
\mathbb{P}\big(|\eta'_{v,z}|\ge x\big)\le2e^{-cx^{2}/\|z\|^{2H}}.\label{eq:957}
\end{equation}
    Indeed, we would like to apply Theorem~\ref{thm: BTIS} with $T:=z+[0,1]^n$ and with the Gaussian process $X_z(s):=-\eta _{v,s}$. This process is almost surely continuous and therefore has continuous covariance function and expectation. Thus, starting with $z=0$ we have by Theorem~\ref{thm: BTIS} that $\mathbb E [\eta '_{v,z}]\ge -C$ for some constant $C$ depending only $n$ and $H$. Moreover, for any other $z\in \mathbb Z ^n$ we have that $\eta '_{v,z}-\eta _{v,z}\overset{d}{=}\eta '_{v,0}$ and therefore $\mathbb E [\eta '_{v,z}]\ge -C$. Moreover, for any $s\in z+[0,1]^n$ we have $\var (X(s)) \le C\|z\|^{2H} $ and therefore \eqref{eq:957} follows from \eqref{eq:theorem5.1}.

It is convenient to extend $\eta '_{v,\cdot }$ from a function on $\mathbb Z ^n$ to a function on $\mathbb R^n$. For any $z\in \mathbb Z ^n$ and $t\in z+[0,1)^n$ we let $\eta '_{v,t}:=\eta '_{v,z}$.

Next, we use some basic facts about sub-Gaussian random variables given in \cite{vershynin2020high}. Recall that the sub-Gaussian norm of a random
variable $Y$ is given by 
\begin{equation}
\|Y\|_{\psi_{2}}\coloneqq\inf\big\{ x>0:\E[\exp(Y^{2}/x^{2})]\le2\big\}.
\end{equation}
(see, e.g., \cite[Definition 2.5.6]{vershynin2020high}). It follows from \cite[Exercise 2.5.7]{vershynin2020high} together with equation \eqref{eq:957} that for any $\|t\|\le 2h$ we have $\|\eta'_{v,t}\|_{\psi_{2}}\le Ch^H$.

Note that for $v\ne w\in\Lambda$ and $s,t\in\R^{n}$, the random
variables $\eta'_{v,s}$ and $\eta'_{w,t}$ are independent and therefore by \cite[Proposition 2.6.1]{vershynin2020high} it follows that
for every $\varphi\in\tilde{\Omega}_{2k}\cap\Omega^{2h}$  we have $\|\sum_{v\in\Lambda}\eta'_{v,\varphi_{v}}\|_{\psi_{2}}^{2}\le Ch^{2H}|\Lambda|$. Hence, by \cite[Equation (2.14)]{vershynin2020high} for every such
$\varphi$ 
\begin{equation}
\begin{split}\P\Big(\sum_{v\in\Lambda}\eta'_{v,\varphi_{v}}\le-k|\Lambda|/4\Big)\le\P\Big(|\sum_{v\in\Lambda}\eta'_{v,\varphi_{v}}|\ge k|\Lambda|/4\Big)\le2\exp\big(-ck^{2}h^{-2H}|\Lambda|\big).
\end{split}
\label{eq:6385}
\end{equation}
We next union bound the last bound over $\varphi\in\Tilde{\Omega}_{2k}$.
Enumerating over the differences between $\varphi_{u}$ and $\varphi_{v}$
for $u\sim v$ we obtain, for $k>C$ 
\begin{equation}
\begin{split}\big|\tilde{\Omega}_{2k}\big| & \le\Big|\Big\{\big(a_{u,v}^{i}\in\mathbb{Z}\ :\ u\sim v,\{u,v\}\cap\Lambda\neq\emptyset,i\in\{1,\dots,n\}\big):\sum|a_{u,v}^{i}|\le2k|\Lambda|\Big\}\Big|\\
 & \le2^{N_{\Lambda}}{2k|\Lambda|+nN_{\Lambda}-1 \choose nN_{\Lambda}-1}\le e^{4k|\Lambda|},
\end{split}
\label{eq:contomegak}
\end{equation}
where $N_{\Lambda}:=\big|\{u\sim v,\{u,v\}\cap\Lambda\neq\emptyset\}\big|$.
Using that $k\ge C_1h^{2H}$ and that $C_1$ can be chosen sufficiently large, we obtain by \eqref{eq:6385} and \eqref{eq:contomegak} that
\begin{equation}
\mathbb{P}\Big(\exists\varphi\in\tilde{\Omega}_{2k}\cap\Omega^{2h}, \ \sum_{v\in\Lambda}\eta'_{v,\varphi_{v}}\le-k|\Lambda|/4\Big)\le2\exp\big(-ck^{2}h^{-2H}|\Lambda|\big).\label{eq:927}
\end{equation}
We claim that \eqref{eq:926} follows from \eqref{eq:927}. Indeed,
for any $\varphi\in\Omega_{k}\cap \Omega ^h$, we consider the function $\tilde{\varphi}$
defined by $\tilde{\varphi}_{v}:=\lfloor\varphi_{v}\rfloor$ (in here
the floor sign of a vector $t\in\mathbb{R}^{n}$ is taken in each
coordinate). For such $\varphi $ we have that
\begin{equation}
\|\nabla\tilde{\varphi}\|_{\Lambda}^{2}\le(\|\nabla\varphi\|_{\Lambda}+\|\nabla(\tilde{\varphi}-\varphi)\|_{\Lambda})^{2}\le\big(\sqrt{k|\Lambda|}+\sqrt{2dn|\Lambda|}\big)^{2}\le2k|\Lambda|\label{eq:norml2discretephi}
\end{equation}
and therefore $\tilde{\varphi}\in\tilde{\Omega}_{2k}\cap \Omega^{2h}$. Finally, by
the definition of $\eta'$, for any surface $\varphi$ we
have that $\sum\eta'_{v,\tilde{\varphi}_{v}} \le \sum\eta_{v,\varphi_{v}}$ and therefore \eqref{eq:926} follows from \eqref{eq:927}. 
\end{proof}

\subsection{Uniqueness}\label{sec:uniqueness}

In order to prove Proposition~\ref{prop:existence} it remains to prove the following claim. The proof of this claim is identical to the uniqueness proof of the minimal surface in the $\eta ^{\rm white}$ environment. See \cite[Lemma~5.4]{dembin2024minimal}.

\begin{claim}
    The minimal surface $\varphi ^{\eta ,\Lambda }$ is almost surely unique. Namely, there is a unique $\varphi $ for which $\GE ^{\eta ,\Lambda }=H^{\eta,\Lambda } (\varphi )$.
\end{claim}

\begin{proof}
Define the following sets for all $v\in\Lambda$, $i\le n$ and
$q\in\mathbb{R}$ 
\begin{equation}
\Omega_{-}(v,i,q):=\Omega^{\Lambda}\cap\big\{\varphi:\mathbb{Z}^{d}\to\mathbb{R}^{n}:(\varphi_v)_{i}\le q\big\},
\end{equation}
\begin{equation}
\Omega_{+}(v,i,q):=\Omega^{\Lambda}\cap\big\{\varphi:\mathbb{Z}^{d}\to\mathbb{R}^{n}:(\varphi_v)_{i}\ge q\big\}.
\end{equation}
Next, define $\varphi_{-}(v,i,q)$ and $\varphi_{+}(v,i,q)$ to be
the the functions $\varphi$ that minimize the Hamiltonian and restricted
to be in $\Omega_{+}(v,i,q)$ and $\Omega_{-}(v,i,q)$ respectively.
These functions exist by the same arguments as in Corollary \ref{cor:maxphi}. We also let
$\GE_{-}(v,i,q):=H(\varphi_{-}(v,i,q))$ and $\GE_{+}(v,i,q):=H(\varphi_{+}(v,i,q))$
. We claim that for all $q_{1}<q_{2}$ we have 
\begin{equation}
\mathbb{P}\big(\GE_{-}(v,i,q_{1})=\GE_{+}(v,i,q_{2})\big)=0.\label{eq:1}
\end{equation}
Indeed, note that the random variable $\GE_{-}(v,i,q_{1})$ is measurable
in 
\begin{equation}
\mathcal{F}(v,i,q_{1}):=\sigma\big(\big\{\eta_{u,t}:u\neq v,t\in\mathbb{R}^{n}\big\}\cup\big\{\eta_{v,t}:t_{i}\le q_{1}\big\}\big).
\end{equation}
Moreover, conditioning on $\mathcal{F}(v,i,q_{1})$, it is not hard to check
that $\GE_{+}(v,i,q_{1})$ has a continuous distribution. This finishes the proof of \eqref{eq:1}.

Next, define the event 
\begin{equation}
\mathcal{A}:=\bigcup_{v\in\Lambda}\bigcup_{i\le n}\bigcup_{\substack{q_{1},q_{2}\in\mathbb{Q}\\
q_{1}<q_{2}
}
}\big\{\GE_{-}(v,i,q_{1})=\GE_{+}(v,i,q_{2})\big\}.
\end{equation}
By \eqref{eq:1} and a union bound we have that $\mathbb{P}(\mathcal{A})=0$.
This finishes the proof of the lemma since on $\mathcal{A}^{c}$ we
cannot have two distinct minimizers to the Hamiltonian. 
\end{proof}

\subsection{Concentration}

We finish the section with a proof of Proposition~\ref{as:conc}.

\begin{proof}[Proof of Proposition~\ref{as:conc}]
    Let $\Delta\subset \Lambda$ and let $h,r> 0$. Recall the definition of  $\Omega ^{\Lambda ,\tau }$  from \eqref{set}. Consider the set of surfaces 
        \begin{equation}
        \Omega _{\Delta,h}:=\Big\{ \varphi \in \Omega ^{\Lambda ,\tau } : \frac{1}{|\Delta |} \sum _{v\in \Delta } \|\varphi _v\| ^{2H} \le h^{2H} \Big\}.
    \end{equation}
    Define also the restricted ground state $\GE _{\Delta,h}$ by 
    \begin{equation}
        \GE  _{\Delta,h}:=\min \big\{ H^{\eta, \Lambda } (\varphi ) : \varphi \in \Omega _{\Delta,h} \big\}.
    \end{equation}

    First, we would like to use Theorem~\ref{thm: BTIS} in order to prove that
    \begin{equation}\label{eq:concentration}
        \P\Big(  \big| \GE _{\Delta,h}-\E [ \GE _{\Delta,h} \mid \eta_{\Delta ^c}] \big| \ge r\ \big| \  \eta_{\Delta^c}\Big)\le 2\exp \Big(- \frac{cr^2}{h^{2H}|\Delta|} \Big).
    \end{equation}
    Indeed, conditioning on $\eta _{\Delta ^c}$, we consider the Gaussian process defined by $X_\varphi :=H^{\eta }(\varphi )$ on the index set $\varphi \in T:=\Omega _{\Delta ,h}$. For any function $\varphi \in \Omega _{\Delta ,h}$ we have that 
    \begin{equation}
        \var \big( H^\eta(\varphi )  \mid \eta_{\Delta^c} \big) = \sum _{v\in \Delta } \|\varphi _v\|^{2H}\le |\Delta | h^{2H}.
    \end{equation}
    Thus, the bound in \eqref{eq:concentration} follows from Theorem~\ref{thm: BTIS}.

    Next, for any fixed $\gamma \in \mathbb R $ we have that 
\begin{equation*}
    \P \big( |\GE^{\eta} - \gamma  | \ge r \mid \eta_{\Delta^c} \big) \le \P \big( \big|\GE_{\Delta,h} - \gamma  \big| \ge r \mid \eta_{\Delta^c} \big) +\mathbb P \big( \GE ^\eta \neq \GE _{\Delta ,h} \mid \eta_{\Delta^c} \big)
\end{equation*}
and therefore
\begin{equation*}
\begin{split}
\inf _{\gamma \in \mathbb R }\P \big( |\GE^{\eta} - \gamma  | \ge r \mid \eta_{\Delta^c} \big) &\le \inf _{\gamma \in \mathbb R } \P \big( \big|\GE_{\Delta,h} - \gamma  \big| \ge r \mid \eta_{\Delta^c} \big) +\mathbb P \big( \GE ^\eta \neq \GE _{\Delta ,h} \mid \eta_{\Delta^c} \big)\\
    &\le \P \big( \big|\GE_{\Delta,h} - \mathbb E [ \GE _h \mid \eta_{\Delta^c} ]  \big| \ge r \mid \eta_{\Delta^c} \big) +\mathbb P \big( \GE ^\eta \neq \GE _{\Delta ,h} \mid \eta_{\Delta^c} \big)\\
    &\le 2\exp \Big(- \frac{cr^2}{h^{2H}|\Delta|} \Big) +\mathbb P \Big(  \frac{1}{|\Delta |} \sum _{v\in \Delta } \|\varphi _v\| ^{2H} \ge h^{2H} \mid \eta_{\Delta^c} \Big),
\end{split}
\end{equation*}
where in the last inequality we used \eqref{eq:concentration} and the definition of $\GE _{\Delta ,h}$.
\end{proof}

\appendix
\section{Basic estimates}

\begin{proof}[Proof of Claim \ref{claim:estimate}]
We first consider the case $\|t\|\le \frac 32$. Note that for $x>0$, since $H\in(0,1)$, we have $x^H\ge \min(1,x)$. It yields
\begin{equation*}
    \begin{split}
        2\|t\|^{2H}+2-\|t-e_1\|^{2H}-\|t+e_1\|^{2H}&= 2\|t\|^{2H}+2-(\|t\|^2-2t_1+1)^H-(\|t\|^2+2t_1+1)^H\\
        &\le 2\|t\|^{2H}+2-(\min(1,\|t\|^2-2t_1+1)+ \min(1,\|t\|^2+2t_1+1))
    \end{split}
\end{equation*}
If $\|t\|^2-2|t_1|\ge 0$, then we can upperbound the right hand side by $2\|t\|^{2H}$, otherwise we get
\begin{equation}
    2\|t\|^{2H}+2-\|t-e_1\|^{2H}-\|t+e_1\|^{2H}\le 2\|t\|^{2H}-\|t\|^2+2|t_1|\le 2(\|t\|^{2H}+\|t\|). 
\end{equation}

Let us now consider the case $\|t\|\ge \frac 32$.
 Define $f:\R\to \R$ to be the function 
   \[f(s)\coloneqq\|t-se_1\|^{2H}.\]
   In particular,
   \[f'(s)= -2(t_1-s)H\|t-se_1\|^{2(H-1)}\]
   and
   \[f''(s)=2H\|t-se_1\|^{2(H-1)}+ 4H(H-1)(t_1-s)^2\|t-se_1\|^{2(H-2)}.\]
   Hence 
   \[|  f''(s)|\le \max(2H\|t-se_1\|^{2(H-1)}, |4H(H-1)|(t_1-s)^2\|t-se_1\|^{2(H-2)}) \le 2\|t-se_1\|^{2(H-1)}.\]
   We have for $\|t \|\ge \frac 32$
   \begin{equation*} 
   \begin{split}
   2\|t\|^{2H}-\|t- e_1\|^{2H}-\|t+ e_1\|^{2H}&=2f(0)-f(1)-f(-1)\le \sup_{s\in[-1,1]}|f''(s)|
   \\&\le  2\sup_{s\in[-1,1]}\|t-se_1\|^{2(H-1)}\le 2\left(\frac 1 2\right)^{2(H-1)}\le 8
   \end{split}
       \end{equation*}
       and \eqref{eq:claim 1} follows by noting that $2((3/2)^{2H}+3/2)\le 8$.
Let us now prove \eqref{eq:claim 2} 
\begin{equation*}
    \begin{split}
        2\|t\|^{2H}+2-\|t-e_1\|^{2H}-\|t+e_1\|^{2H}&=2\|t\|^{2H}+2- (\|t\|^2+1+2t_1)^H-(\|t\|^2+1-2t_1)^H\\&\ge 2\|t\|^{2H}+2- 2(\|t\|^2+1)^H
    \end{split}
\end{equation*}
where we used in the last inequality the concavity of the function $x\mapsto x^H $. Finally, using that the function $g:x\mapsto x^H+1-(x+1)^H$ is increasing on $(0,+\infty)$ we obtain for $H\in(0,1)$
\begin{equation}\label{eq:lower bound cov}
    2\|t\|^{2H}+2-\|t-e_1\|^{2H}-\|t+e_1\|^{2H}\ge 2g(\|t\|)\ge 2g(0)= 0.
\end{equation}
\end{proof}

\bibliographystyle{plain}
\bibliography{MSRE}

\begin{thebibliography}{10}

\bibitem{arous2024free}
G{\'e}rard~Ben Arous and Pax Kivimae.
\newblock The free energy of the elastic manifold.
\newblock {\em arXiv preprint arXiv:2410.19094}, 2024.

\bibitem{arous2024larkin}
G{\'e}rard~Ben Arous and Pax Kivimae.
\newblock The larkin mass and replica symmetry breaking in the elastic manifold.
\newblock {\em arXiv preprint arXiv:2410.22601}, 2024.

\bibitem{bakhtin2016inviscid}
Yuri Bakhtin.
\newblock Inviscid {B}urgers equation with random kick forcing in noncompact setting.
\newblock {\em Electronic Journal of Probability}, 21, 2016.

\bibitem{bakhtin2014space}
Yuri Bakhtin, Eric Cator, and Konstantin Khanin.
\newblock Space-time stationary solutions for the {B}urgers equation.
\newblock {\em Journal of the American Mathematical Society}, 27(1):193--238, 2014.

\bibitem{bakhtin2022dynamic}
Yuri Bakhtin and Hong-Bin Chen.
\newblock Dynamic polymers: invariant measures and ordering by noise.
\newblock {\em Probability Theory and Related Fields}, 183(1):167--227, 2022.

\bibitem{bakhtin2018zero}
Yuri Bakhtin and Liying Li.
\newblock Zero temperature limit for directed polymers and inviscid limit for stationary solutions of stochastic {B}urgers equation.
\newblock {\em Journal of Statistical Physics}, 172(5):1358--1397, 2018.

\bibitem{bakhtin2019thermodynamic}
Yuri Bakhtin and Liying Li.
\newblock Thermodynamic limit for directed polymers and stationary solutions of the {B}urgers equation.
\newblock {\em Communications on Pure and Applied Mathematics}, 72(3):536--619, 2019.

\bibitem{balents1993large}
Leon Balents and Daniel~S Fisher.
\newblock Large-n expansion of (4-$\varepsilon$)-dimensional oriented manifolds in random media.
\newblock {\em Physical Review B}, 48(9):5949, 1993.

\bibitem{ben2024landscape}
G{\'e}rard Ben~Arous, Paul Bourgade, and Benjamin McKenna.
\newblock Landscape complexity beyond invariance and the elastic manifold.
\newblock {\em Communications on Pure and Applied Mathematics}, 77(2):1302--1352, 2024.

\bibitem{berger2019entropy}
Quentin Berger and Niccol{\`o} Torri.
\newblock {Entropy-controlled Last-Passage Percolation}.
\newblock {\em The Annals of Applied Probability}, 29(3):1878 -- 1903, 2019.

\bibitem{berger2021beyond}
Quentin Berger and Niccol{\`o} Torri.
\newblock Beyond hammersley’s last-passage percolation: a discussion on possible local and global constraints.
\newblock {\em Annales de l’Institut Henri Poincar{\'e} D}, 8(2):213--241, 2021.

\bibitem{colding2011course}
Tobias~H Colding and William~P Minicozzi.
\newblock {\em A course in minimal surfaces}, volume 121.
\newblock American Mathematical Soc., 2011.

\bibitem{comets2018brownian}
Francis Comets and Cl{\'e}ment Cosco.
\newblock Brownian polymers in {P}oissonian environment: A survey.
\newblock {\em arXiv preprint arXiv:1805.10899}, 2018.

\bibitem{dario2023random}
Paul Dario, Matan Harel, and Ron Peled.
\newblock Random-field random surfaces.
\newblock {\em Probability Theory and Related Fields}, pages 1--68, 2023.

\bibitem{dario2024quantitative}
Paul Dario, Matan Harel, and Ron Peled.
\newblock Quantitative disorder effects in low-dimensional spin systems.
\newblock {\em Communications in Mathematical Physics}, 405(9):212, 2024.

\bibitem{de2022regularity}
Camillo De~Lellis.
\newblock The regularity theory for the area functional (in geometric measure theory).
\newblock In {\em International Congress of Mathematicians}, 2022.

\bibitem{dembin2024minimal}
Barbara Dembin, Dor Elboim, Daniel Hadas, and Ron Peled.
\newblock Minimal surfaces in random environment.
\newblock {\em arXiv preprint arXiv:2401.06768}, 2024.

\bibitem{emig1998roughening}
Thorsten Emig and Thomas Nattermann.
\newblock Roughening transition of interfaces in disordered systems.
\newblock {\em Physical review letters}, 81(7):1469, 1998.

\bibitem{emig1999disorder}
Thorsten Emig and Thomas Nattermann.
\newblock Disorder driven roughening transitions of elastic manifolds and periodic elastic media.
\newblock {\em The European Physical Journal B-Condensed Matter and Complex Systems}, 8(4):525--546, 1999.

\bibitem{ferrero2021creep}
Ezequiel~E Ferrero, Laura Foini, Thierry Giamarchi, Alejandro~B Kolton, and Alberto Rosso.
\newblock Creep motion of elastic interfaces driven in a disordered landscape.
\newblock {\em Annual Review of Condensed Matter Physics}, 12:111--134, 2021.

\bibitem{forgacs1991behavior}
Gabor Forgacs, Reinhard Lipowsky, and Theo~M Nieuwenhuizen.
\newblock The behavior of interfaces in ordered and disordered systems.
\newblock {\em Phase transitions and critical phenomena}, 14:135--363, 1991.

\bibitem{giamarchi2009disordered}
Thierry Giamarchi.
\newblock Disordered elastic media.
\newblock {\em Encyclopedia of complexity and systems science}, 112:2019--2038, 2009.

\bibitem{giamarchi1998statics}
Thierry Giamarchi and Pierre Le~Doussal.
\newblock Statics and dynamics of disordered elastic systems.
\newblock In {\em Spin glasses and random fields}, pages 321--356. World Scientific, 1998.

\bibitem{grinstein1983surface}
G~Grinstein and Shang-keng Ma.
\newblock Surface tension, roughening, and lower critical dimension in the random-field ising model.
\newblock {\em Physical Review B}, 28(5):2588, 1983.

\bibitem{grinstein1982roughening}
Geoffrey Grinstein and Shang-keng Ma.
\newblock Roughening and lower critical dimension in the random-field {I}sing model.
\newblock {\em Physical Review Letters}, 49(9):685, 1982.

\bibitem{halpin1989diverse}
Timothy Halpin-Healy.
\newblock Diverse manifolds in random media.
\newblock {\em Physical review letters}, 62(4):442, 1989.

\bibitem{halpin1990disorder}
Timothy Halpin-Healy.
\newblock Disorder-induced roughening of diverse manifolds.
\newblock {\em Physical Review A}, 42(2):711, 1990.

\bibitem{Kallenberg}
Olav Kallenberg.
\newblock {\em Foundations of modern probability}, volume~99 of {\em Probability Theory and Stochastic Modelling}.
\newblock Springer, Cham, third edition, 2021.

\bibitem{kardar1987domain}
Mehran Kardar.
\newblock Domain walls subject to quenched impurities.
\newblock {\em Journal of Applied Physics}, 61(8):3601--3604, 1987.

\bibitem{larkin1970effect}
Anatoly~I Larkin.
\newblock Effect of inhomogeneties on the structure of the mixed state of superconductors.
\newblock {\em Soviet Journal of Experimental and Theoretical Physics}, 31:784, 1970.

\bibitem{meeks2012survey}
William Meeks and Joaqu{\'\i}n P{\'e}rez.
\newblock {\em A survey on classical minimal surface theory}, volume~60.
\newblock American Mathematical Soc., 2012.

\bibitem{mezard1990interfaces}
Marc M{\'e}zard and Giorgio Parisi.
\newblock Interfaces in a random medium and replica symmetry breaking.
\newblock {\em Journal of Physics A: Mathematical and General}, 23(23):L1229, 1990.

\bibitem{mezard1991replica}
Marc M{\'e}zard and Giorgio Parisi.
\newblock Replica field theory for random manifolds.
\newblock {\em Journal de Physique I}, 1(6):809--836, 1991.

\bibitem{mezard1992manifolds}
Marc M{\'e}zard and Giorgio Parisi.
\newblock Manifolds in random media: two extreme cases.
\newblock {\em Journal de Physique I}, 2(12):2231--2242, 1992.

\bibitem{nattermann1988random}
T~Nattermann and J~Villain.
\newblock Random-field {I}sing systems: A survey of current theoretical views.
\newblock {\em Phase transitions}, 11(1-4):5--51, 1988.

\bibitem{nattermann1987interface}
Thomas Nattermann.
\newblock Interface roughening in systems with quenched random impurities.
\newblock {\em EPL (Europhysics Letters)}, 4(11):1241, 1987.

\bibitem{ossiander1989certain}
Mina Ossiander and Edward~C Waymire.
\newblock Certain positive-definite kernels.
\newblock {\em Proceedings of the American Mathematical Society}, 107(2):487--492, 1989.

\bibitem{otto2025minimizing}
Felix Otto, Matteo Palmieri, and Christian Wagner.
\newblock On minimizing curves in a brownian potential.
\newblock {\em arXiv preprint arXiv:2503.12471}, 2025.

\bibitem{vershynin2020high}
Roman Vershynin.
\newblock {\em High-dimensional probability}.
\newblock Cambridge university press, 2018.

\bibitem{villain1982commensurate}
J~Villain.
\newblock Commensurate-incommensurate transition with frozen impurities.
\newblock {\em Journal de Physique Lettres}, 43(15):551--558, 1982.

\bibitem{wiese2003functional}
Kay~J{\"o}rg Wiese.
\newblock The functional renormalization group treatment of disordered systems, a review.
\newblock In {\em Annales Henri Poincar{\'e}}, volume 4, Suppl 1, pages S505--S528. Springer, 2003.

\bibitem{wiese2022theory}
Kay~J{\"o}rg Wiese.
\newblock Theory and experiments for disordered elastic manifolds, depinning, avalanches, and sandpiles.
\newblock {\em Reports on Progress in Physics}, 85(8):086502, 2022.

\bibitem{zeitouni2014gaussian}
Ofer Zeitouni.
\newblock Gaussian fields.
\newblock \url{https://www.wisdom.weizmann.ac.il/~zeitouni/notesGauss.pdf}, 2014.

\end{thebibliography}

\end{document}